\documentclass[12pt]{amsart}
\usepackage{amsfonts, amssymb, latexsym, epsfig, amscd, ytableau, multicol, tikz}
\usepackage{hyperref,graphicx}

\usepackage{wrapfig}
\newtheorem{Theorem}{Theorem}[section]

\newtheorem{Proposition}[Theorem]{Proposition}
\newtheorem{Lemma}[Theorem]{Lemma}

\newtheorem{Corollary}[Theorem]{Corollary}
\newtheorem{Problem}[Theorem]{Problem}
\newtheorem{Definition-Proposition}[Theorem]{Definition-Proposition}
\newtheorem{Main Conjecture}[Theorem]{Main Conjecture}

\newtheorem{Conjecture}[Theorem]{Conjecture}
\newtheorem{Definition}[Theorem]{Definition}
\newtheorem{Remark}[Theorem]{Remark}
\theoremstyle{remark}
\newtheorem{Example}[Theorem]{Example}

\newcommand{\excise}[1]{}

\theoremstyle{plain}

\newcommand{\cellsize}{11}
\newlength{\cellsz} 
\newsavebox{\cell}
\sbox{\cell}{\begin{picture}(\cellsize,\cellsize)
\put(0,0){\line(1,0){\cellsize}}
\put(0,0){\line(0,1){\cellsize}}
\put(\cellsize,0){\line(0,1){\cellsize}}
\put(0,\cellsize){\line(1,0){\cellsize}}
\end{picture}}
\newcommand\cellify[1]{\def\thearg{#1}\def\nothing{}%
\ifx\thearg\nothing
\vrule width0pt height\cellsz depth0pt\else
\hbox to 0pt{\usebox{\cell} \hss}\fi%
\vbox to \cellsz{
\vss
\hbox to \cellsz{\hss$#1$\hss}
\vss}}
\newcommand\tableau[1]{\vtop{\let\\\cr
\baselineskip -16000pt \lineskiplimit 16000pt \lineskip 0pt
\ialign{&\cellify{##}\cr#1\crcr}}}
\newcommand{\kellsize}{24}
\newlength{\kellsz} 
\newsavebox{\kell}
\sbox{\kell}{\begin{picture}(\kellsize,\kellsize)
\put(0,0){\line(1,0){\kellsize}}
\put(0,0){\line(0,1){\kellsize}}
\put(\kellsize,0){\line(0,1){\kellsize}}
\put(0,\kellsize){\line(1,0){\kellsize}}
\end{picture}}
\newcommand\kellify[1]{\def\thearg{#1}\def\nothing{}%
\ifx\thearg\nothing
\vrule width0pt height\kellsz depth0pt\else
\hbox to 0pt{\usebox{\kell} \hss}\fi%
\vbox to \kellsz{
\vss
\hbox to \kellsz{\hss$#1$\hss}
\vss}}
\newcommand\ktableau[1]{\vtop{\let\\\cr
\baselineskip -16000pt \lineskiplimit 16000pt \lineskip 0pt
\ialign{&\kellify{##}\cr#1\crcr}}}

\newcommand{\sellsize}{36}
\newlength{\sellsz} 
\newsavebox{\sell}
\sbox{\sell}{\begin{picture}(\sellsize,20)
\put(0,0){\line(1,0){\sellsize}}
\put(0,0){\line(0,1){\sellsize}}
\put(\sellsize,0){\line(0,1){\sellsize}}
\put(0,\sellsize){\line(1,0){\sellsize}}
\end{picture}}
\newcommand\sellify[1]{\def\thearg{#1}\def\nothing{}%
\ifx\thearg\nothing
\vrule width0pt height\sellsz depth0pt\else
\hbox to 0pt{\usebox{\sell} \hss}\fi%
\vbox to \sellsz{
\vss
\hbox to \sellsz{\hss$#1$\hss}
\vss}}
\newcommand\stableau[1]{\vtop{\let\\\cr
\baselineskip -16000pt \lineskiplimit 16000pt \lineskip 0pt
\ialign{&\sellify{##}\cr#1\crcr}}}
\usepackage{array,multirow}
\begin{document}
\pagestyle{plain}

\mbox{}
\title{Newton Polytopes in algebraic combinatorics}
\author{Cara Monical}
\author{Neriman Tokcan}
\author{Alexander Yong}
\address{Dept.~of Mathematics, U.~Illinois at
Urbana-Champaign, Urbana, IL 61801, USA}
\email{cmonica2@illinois.edu, tokcan2@illinois.edu, ayong@uiuc.edu}

\date{\today}

\maketitle

\begin{abstract}
A polynomial has \emph{saturated Newton polytope} (SNP) 
if every lattice point of the convex hull of its exponent vectors
corresponds to a monomial. We compile instances of SNP 
in algebraic combinatorics (some with proofs, others conjecturally): skew Schur polynomials; 
symmetric polynomials associated to reduced words, Redfield--P\'olya theory, Witt vectors, and totally nonnegative
matrices; resultants;
discriminants (up to quartics); Macdonald polynomials; key polynomials; 
Demazure atoms; Schubert polynomials; and Grothendieck polynomials, among others.

Our principal construction is 
the \emph{Schubitope}. For any subset of $[n]^2$,
we describe it by linear inequalities. This generalized permutahedron conjecturally
has positive Ehrhart polynomial. We conjecture it describes the Newton polytope  of Schubert and key polynomials.
We also define \emph{dominance order} on permutations and study its poset-theoretic properties.
\end{abstract}

\section{Introduction}
The {\bf Newton polytope} of 
a polynomial $f=\sum_{\alpha\in {\mathbb Z}_{\geq 0}^n}c_{\alpha} x^{\alpha}\in {{\mathbb C}}[x_1,\ldots,x_n]$ is 
the convex hull of its exponent vectors, i.e., 
\[{\sf Newton}(f)={\rm conv}(\{\alpha:c_{\alpha}\neq 0\})\subseteq {\mathbb R}^n.\]
\begin{Definition}
$f$ has {\bf saturated Newton polytope} (SNP) if $c_{\alpha}\neq 0$ whenever $\alpha\in {\sf Newton}(f)$.
\end{Definition}

\begin{Example}
$f=$ the determinant of a
generic $n\times n$ matrix. The exponent vectors
correspond to permutation matrices. ${\sf Newton}(f)$ is the 
\emph{Birkhoff polytope} of $n\times n$ doubly stochastic matrices. SNPness says there are no
additional lattice points, which is obvious here. (The Birkhoff-von Neumann theorem states all lattice points are 
vertices.) \qed
\end{Example}

Generally, polynomials are not SNP. Worse still, SNP is not preserved by basic 
polynomial operations. For example, $f=x_1^2+x_2x_3+x_2x_4+x_3x_4$ is SNP but $f^2$ is 
not (it misses $x_1 x_2 x_3 x_4$). Nevertheless, there are a number of 
families of polynomials in algebraic combinatorics where every member is 
(conjecturally) SNP. Examples motivating our investigation include:

\begin{itemize}
\item The \emph{Schur polynomials} are SNP. This rephrases
R.~Rado's theorem \cite{rado} about permutahedra and
dominance order on partitions; cf. Proposition~\ref{prop:generalthing}.
\item Classical \emph{resultants} are SNP (Theorem~\ref{prop:resultantSNP}). Their Newton polytopes were studied by I.~M.~Gelfand-M.~Kapranov-A.~Zelevinsky \cite{GKZ}. (Classical \emph{discriminants} are SNP up to quartics --- but not quintics; see Proposition~\ref{abel}.) 
\item \emph{Cycle index polynomials} from Redfield--P\'{o}lya theory (Theorem~\ref{thm:cycleindex})
\item C.~Reutenauer's symmetric polynomials linked to the 
free Lie algebra and
to Witt vectors \cite{Reutenauer} (Theorem~\ref{thm:Reutenauer})
\item J.~R.~Stembridge's symmetric polynomials associated to 
totally nonnegative matrices \cite{Stembridge}  (Theorem~\ref{thm:Stem})
\item R.~P.~Stanley's symmetric polynomials \cite{Stanley1984}, 
introduced to enumerate reduced words of permutations
(Theorem~\ref{thm:Stanley})
\item Any generic $(q,t)$-evaluation of a \emph{symmetric Macdonald polynomial} is SNP; see 
Theorem~\ref{prop:PisSNP} and Proposition~\ref{prop:modifiedisSNP}. 
\item The \emph{key polynomials} are  
$(\infty,\infty)$-specializations of \emph{non-symmetric Macdonald polynomials}. These also seem to be SNP. We give two  conjectural descriptions of the Newton polytopes. We determine a list of vertices
of the Newton polytopes (Theorem~\ref{cor:vertexkeys}) and conjecture this list is complete (Conjecture~\ref{conj:vertexkeys}). 
\item \emph{Schubert polynomials} (Conjecture~\ref{conj:main2}). We
conjecturally describe the Newton polytope (Conjecture~\ref{conj:main1}). 
\item Inhomogeneous versions of Schuberts/keys are also conjecturally SNP (Conjectures~\ref{conj:grothendieckSNP} and~\ref{conj:mar7bhh}).
\end{itemize}

The core part of our study concerns the Schubert and key polynomials. We
conjecture a  description of their 
Newton polytopes in terms of a new family
of polytopes.

\begin{wrapfigure}{l}{0.12\textwidth}
\begin{tikzpicture}[scale=0.5]
\draw (0,0) rectangle (4,4);

\draw[fill=gray!30] (2,3) rectangle (3,4);
\draw[fill=gray!30] (1,2) rectangle (2,3);
\draw[fill=gray!30] (0,1) rectangle (1,2);
\draw[fill=gray!30] (1,1) rectangle (2,2);
\draw[fill=gray!30] (3,1) rectangle (4,2);
\draw[fill=gray!30] (2,0) rectangle (3,1);
\end{tikzpicture}
\end{wrapfigure}

A {\bf diagram} ${\sf D}$ is a subset boxes of an $n \times n$ grid. 
Fix $S\subseteq [n]:=\{1,2,\ldots,n\}$ and a column $c \in [n]$. Let ${\tt word}_{c,S}({\sf D})$ be formed by reading $c$ from top to bottom and recording 
\begin{itemize} 
\item $($ if $(r,c) \notin {\sf D}$ and $r \in S$,  
\item $)$ if $(r,c) \in {\sf D}$ and $r \notin S$, and
\item $\star$ if $(r,c) \in {\sf D}$ and $r \in S$.
\end{itemize}
Let 
\[\theta_{\sf D}^c(S)= \#\text{paired $(\ )$'s in ${\tt word}_{c,S}({\sf D})$} + 
\#\text{$\star$'s in ${\tt word}_{c,S}({\sf D})$}.\]  
Set 
$\theta_{\sf D}(S) = \sum_{c \in [n]} \theta_{\sf D}^c(S)$. For instance,
$\theta_{\sf D}(\{ 2,4 \}) = 4$ above.
Define the {\bf Schubitope} as  
\[{\mathcal S}_{\sf D}=\left\{(\alpha_1,\ldots,\alpha_n)\in {\mathbb R}_{\geq 0}^n:
\sum_{i=1}^n \alpha_i=\#{\sf D} \text{\ and \ }
\sum_{i\in S}\alpha_i \leq \theta_{\sf D}(S) \text{\ for all $S \subset [n]$}\right\}.\]

\begin{wrapfigure}{l}{0.46\textwidth}
\begin{picture}(200,200)
\put(0,0){\includegraphics[height=2.9in]{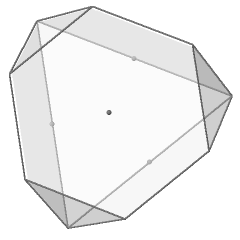}}
\put(35,-2){$(3,1,0,0)$}
\put(105,4){$(2,0,2,0)$}
\put(160,65){$(2,2,0,0)$}
\put(170,115){$(1,1,1,1)$}
\put(145,175){$(3,0,1,0)$}
\put(55,210){$(2,1,1,0)$}
\put(-15,190){$(1,0,2,1)$}
\put(-25,130){$(1,2,1,0)$}
\put(-15,35){$(1,1,2,0)$}
\put(70,95){$(1,2,0,1)$}
\put(-10,95){\color{gray}$(2,0,1,1)$}
\put(80,65){\color{gray}$(2,1,0,1)$}
\put(65,155){\color{gray}$(3,0,0,1)$}
\end{picture}
\hspace{-.2in}
\end{wrapfigure}
\noindent
Fix a partition $\lambda=(\lambda_1,\lambda_2,\ldots,\lambda_n)$.
The {\bf $\lambda$-permutahedron}, denoted ${\mathcal P}_{\lambda}$, is the convex hull of the 
$S_n$-orbit
of $\lambda$ in ${\mathbb R}^n$. The Schubitope is a generalization of the
permutahedron (Proposition~\ref{thm:first}). We conjecture that the Schubitope for a \emph{skyline diagram} and
for a \emph{Rothe diagram} respectively are the Newton polytopes of a key and Schubert polynomial.
The figure to the left depicts ${\mathcal S}_{{\sf D}_{21543}}$, which is a three-dimensional convex polytope in ${\mathbb R}^4$.
Conjecture~\ref{conj:ehrhart} asserts that \emph{Ehrhart polynomials} of Schubitopes
${\mathcal S}_{{\sf D}_w}$ have positive coefficients; 
cf.~\cite[Conjecture~1.2]{Castillo.Liu}.

A cornerstone of the theory of symmetric polynomials is the
combinatorics of \emph{Littlewood-Richardson coefficients}. 
An important special case of these numbers are the
\emph{Kostka coefficients} ${\sf K}_{\lambda,\mu}$. 
The nonzeroness of ${\sf K}_{\lambda,\mu}$ is governed by \emph{dominance order} which is defined by the linear
inequalities (\ref{eqn:domorderineq}). Alternatively, Rado's theorem \cite[Theorem~1]{rado}
states this order characterizes when ${\mathcal P}_{\mu}\subseteq {\mathcal P}_{\lambda}$. These two viewpoints on dominance order are connected since 
${\mathcal P}_{\lambda}$ is the Newton polytope
of the Schur polynomial $s_{\lambda}(x_1,x_2,\ldots,x_n)$.

For Schubert polynomials, there is no analogous Littlewood-Richardson
rule. However, with a parallel in mind,  we propose a ``dominance order for permutations'' \emph{via} Newton polytopes. The inequalities of the Schubitope generalize (\ref{eqn:domorderineq}); see Proposition~\ref{thm:first}.

\subsection*{Organization}
Section~\ref{sec:symmetricfunctions} develops and applies basic results about SNP symmetric polynomials. 
Section~\ref{sec:mac} turns to flavors of Macdonald polynomials and their specializations, including the
key polynomials and \emph{Demazure atoms}. 
Section~\ref{sec:quasi} concerns quasisymmetric functions.
\emph{Monomial quasisymmetric} and \emph{Gessel's fundamental quasisymmetric polynomials} are not SNP, but have
equal Newton polytopes. The \emph{quasisymmetric Schur polynomials} \cite{quasischur} are also not SNP, which
demonstrates a qualitative difference with Schur polynomials.
Section~\ref{sec:schub} discusses Schubert polynomials and a number of
variations. We define dominance order
for permutations and study its poset-theoretic
properties. We connect the Schubitope to work of A.~Kohnert
\cite{Kohnert} and explain a salient contrast (Remark~\ref{remark:diffwithkohnert}). 

\section{Symmetric functions}
\label{sec:symmetricfunctions}

\subsection{Preliminaries}
The {\bf monomial symmetric polynomial} for a partition $\lambda$ is
\[m_{\lambda}=\sum_{\alpha} x_1^{\alpha_1}x_2^{\alpha_2}\cdots x_n^{\alpha_n}\]
where the sum is over distinct rearrangements of $\lambda$. The set
$\{m_{\lambda}\}_{\ell(\lambda)\leq n}$ forms a ${\mathbb Z}$-basis of ${\sf Sym}_n$, the ring of symmetric
polynomials in $x_1,x_2,\ldots,x_n$ (here $\ell(\lambda)$ is the number of nonzero parts of $\lambda$).

Identify a partition $\lambda$ with its Young diagram (in English notation). A {\bf semistandard Young tableau} $T$
is a filling of $\lambda$ with entries from ${\mathbb Z}_{>0}$ that is weakly increasing along rows and strictly increasing down
columns. The {\bf content} of $T$ is $\mu=(\mu_1,\mu_2,\ldots,\mu_n)$ where $\mu_i$ is the number of $i$'s appearing
in $T$. The {\bf Schur polynomial} is
\begin{equation}
\label{eqn:schurdef}
s_{\lambda}=\sum_{\mu}{\sf K}_{\lambda,\mu}m_{\mu}
\end{equation}
where ${\sf K}_{\lambda,\mu}$ is the number of semistandard Young tableaux of shape $\lambda$ and content $\mu$.

Let ${\sf Par}(d)=\{\lambda:\lambda\vdash d\}$ be the set of partitions of size $d$. {\bf Dominance order} $\leq_D$ on 
${\sf Par}(d)$ is defined by
\begin{equation}
\label{eqn:domorderineq}
\mu\leq_{D} \lambda \text{\ \ \ if \ \ \ $\sum_{i=1}^k\mu_i\leq \sum_{i=1}^k\lambda_i$ \ \ \ for all $k\geq 1$.}
\end{equation}
Recall this result about tableaux (see e.g., \cite[Proposition~7.10.5 and Exercise~7.12]{Stanley}):
\begin{equation}
\label{eqn:kostkadominance}
\text{${\sf K}_{\lambda,\mu} \neq 0 \iff \mu \leq_D \lambda$.} 
\end{equation}

Since ${\sf K}_{\lambda,\lambda}=1$, it follows 
from (\ref{eqn:schurdef}) and (\ref{eqn:kostkadominance}) combined that $\{s_{\lambda}\}_{\ell(\lambda) \leq n}$ also forms a basis of ${\sf Sym}_n$.

Setting $x_{n+1}=0$ defines a surjective homomorphism ${\sf Sym}_{n+1}\twoheadrightarrow 
{\sf Sym}_n$ for each $n\geq 0$.
Let ${\sf Sym}$ denote $\underset{n}{\varprojlim}\ {\sf Sym}_n$, the ring of symmetric \emph{functions} in $x_1,x_2,\ldots$. We refer the reader to \cite[Chapter~7]{Stanley} for additional background.

\subsection{Basic facts about SNP}
Given $f\in {\sf Sym}$, let $f(x_1,\ldots,x_n)\in {\sf Sym}_n$ 
be the specialization that sets $x_{i}=0$ for $i\geq n+1$. Whether $f(x_1,\ldots,x_n)$ 
is SNP depends on $n$ (e.g., $f=\sum_i x_i^2$). 

\begin{Definition}
$f\in {\sf Sym}$ is SNP if $f(x_1,\ldots,x_m)$ is SNP for all $m\geq 1$.
\end{Definition}

\begin{Proposition}[Stability of SNP]
\label{prop:stability}
Suppose $f\in {\sf Sym}$ has finite degree. Then 
$f$ is SNP if there exists $m\geq \deg(f)$ such that
$f(x_1,\ldots,x_m)$ is SNP.
\end{Proposition}
\begin{proof}
We first show that if $f(x_1,\ldots,x_m)$ is SNP, 
$f(x_1,\ldots,x_n)$ is SNP for any $n \leq m$.  
Suppose $\alpha \in {\tt Newton}(f(x_1,\ldots,x_n))\subseteq {\tt Newton}(f(x_1,\dots,x_m))$.  Since $f(x_1,\ldots,x_m)$ is SNP, $x^\alpha$ is a monomial of $f(x_1,\ldots,x_m)$.  However, since $\alpha \in {\tt Newton}(f(x_1,\ldots,x_n))$, $\alpha$ only uses the first $n$ positions and thus $x^\alpha$ is a monomial of 
$f(x_1,\ldots,x_n,0,\ldots,0)$.

To complete the proof, we now show if $f(x_1,\ldots,x_m)$ for $m \geq \deg(f)$ is SNP, then 
$f(x_1,\ldots,x_n)$ is SNP for all $n \geq m$. 
Suppose $\alpha \in {\tt Newton}(f(x_1,\ldots,x_n))$ and thus 
\begin{equation}
\label{eqn:feb17a}
\alpha = \sum_i c_i \beta^i 
\end{equation}
 where $x^{\beta^i}$ is a monomial of $f$.  Since $m \geq \deg(f)$, there are at most $m$ coordinates where $\alpha_j > 0$, say $j_1,\ldots,j_m$.  
Furthermore, since each $\beta^i$ is nonnegative, if $c_i > 0$, $\beta_j^i = 0$ for $j \neq j_1, \ldots, j_m$.
Choose $w \in S_n$ such that $w(j_c) = c$ for $c = 1,\ldots,m$.  Applying $w$ to (\ref{eqn:feb17a})
gives 
\[ w(\alpha) = \sum_i c_i w(\beta^i).\]
So nonzero coordinates of $w(\alpha)$ only occur in positions $1, \ldots, m$.  Since $f\in {\sf Sym}$, each $x^{w(\beta^i)}$ is a monomial of $f(x_1,\ldots,x_m)$, and so $w(\alpha) \in {\tt Newton}(f(x_1,\ldots,x_m))$.  Since $f(x_1,\ldots,x_m)$ is SNP, $[x^{w(\alpha)}]f\neq 0$.
Again, $f\in {\sf Sym}$ implies $[x^\alpha]f\neq 0$. Hence, $f(x_1,\ldots,x_n)$ is SNP.
\end{proof}

\begin{Remark}
\label{remark:quasistability}
\emph{In the proof of Proposition~\ref{prop:stability}, $w$ is chosen 
so that the nonzero components of the vectors $\alpha$
and $w(\alpha)$ are in the same relative order. Thus the result extends to the quasisymmetric case of Section \ref{sec:quasi}.}\qed
\end{Remark}

\begin{Remark}
\emph{
The stabilization constant $\deg(f)$ is tight. Let $f=s_{\lambda}-{\sf f}^{\lambda}m_{(1^{|\lambda|})}$. Here ${\sf f}^{\lambda}=[m_{(1^{|\lambda|})}]s_{\lambda}$  
(=the number of standard Young tableaux of shape $\lambda$.)
Then $f(x_1,\ldots,x_n)$ is SNP for $n<|\lambda|=\deg(f)$ but not SNP for $n\geq |\lambda|$. One can see this from the ideas in the next proposition.}
\end{Remark}

\begin{Proposition}
\label{prop:generalthing}
Suppose $f\in {\sf Sym}_n$ is homogeneous of degree $d$ such that 
\[f = \sum_{\mu\in {\sf Par}(d)} c_\mu s_\mu.\] 
Suppose there exists $\lambda$ with 
$c_{\lambda}\neq 0$
and $c_\mu \neq 0$ only if 
$\mu \leq_D \lambda$.  If $n < \ell(\lambda)$, $f = 0$.  Otherwise:
\begin{itemize}
\item[(I)] ${\sf Newton}(f) = {\mathcal P}_{\lambda}\subset {\mathbb R}^n$. 
\item[(II)] The vertices of ${\sf Newton}(f)$ are rearrangements of $\lambda$.
\item[(III)] If moreover $c_{\mu}\geq 0$ for all $\mu$, $f$ has SNP.
\end{itemize}
\end{Proposition}
\begin{proof} 
If $\mu \leq_D \lambda$, $\ell(\mu) \geq \ell(\lambda)$. Thus if $n < 
\ell(\lambda)$, 
$s_\mu(x_1,\ldots,x_n) \equiv 0$ for all $\mu$ such that $c_\mu \neq 0$.
Otherwise, suppose $n \geq \ell(\lambda)$.

(I): Since $\displaystyle f = \sum_{\mu \leq_D \lambda} c_\mu s_\mu $, 
by (\ref{eqn:kostkadominance}) we have
\[f=\sum_{\mu\leq_D\lambda} d_{\mu} m_{\mu}.\]

Clearly, 
\begin{equation}
\label{eqn:newtonmonomial}
{\sf Newton}(m_\mu(x_1,\ldots,x_n)) = {\mathcal P}_{\mu}\subset{\mathbb R}^n
\end{equation}
(by the definitions of both). Also, 
\[{\sf Newton}(f+g)={\sf conv}({\sf Newton}(f)\cup 
{\sf Newton}(g)).\]
Hence,
\[ {\sf Newton}(f) = {\rm conv}\left(\bigcup_{\mu \leq_D \lambda} {\sf Newton}(m_\mu)\right) = {\rm conv}\left(\bigcup_{\mu \leq_D \lambda} {\mathcal P}_{\mu}\right). \]
R.~Rado's theorem \cite[Theorem~1]{rado} states:
\begin{equation}
\text{${\mathcal P}_{\mu} \subseteq {\mathcal P}_{\lambda}\iff\mu \leq_D \lambda$.}
\label{eqn:containDominance}
\end{equation}
Now ${\sf Newton}(f)  =  {\mathcal P}_{\lambda}$
holds by (\ref{eqn:containDominance}), proving (I).

(II): In view of (I), it suffices to know this claim for ${\mathcal 
P}_{\lambda}$. This is well-known, but we include a proof
for completeness. 

Since ${\mathcal P}_{\lambda}$ is the convex hull of the $S_n$-orbit of $\lambda$, any vertex of ${\mathcal P}_{\lambda}$ is a rearrangement of $\lambda$.
It remains to show that every such rearrangement $\beta$ is in fact a vertex. Thus it suffices to show
there is no nontrivial convex combination
\begin{equation}
\label{eqn:convecombfeb16}
\beta = \sum_{\gamma} c_\gamma \gamma,
\end{equation}
where the sum is over distinct rearrangements $\gamma \neq \beta$ of $\lambda$.

Let 
$\lambda = (\Lambda_1^{k_1}\cdots \Lambda_m^{k_m})$
with $\Lambda_1 > \Lambda_2 > \ldots > \Lambda_m$.
Since $\beta$ is a rearrangement of $\lambda$, let $i_1^1, \ldots, i_{k_1}^1$ be the positions 
in $\beta$ of the $k_1$ parts of size $\Lambda_1$.  
Since $\gamma_{i_j^1}\leq \Lambda_1$ for all $\gamma$ we have that $c_{\gamma} = 0$ 
whenever  $\gamma$ satisfies $\gamma_{i_j^1}\neq \Lambda_1$ for any $1\leq j\leq k_1$.

Let $i_1^2, \ldots, i_{k_2}^2$ be the positions in $\beta$ 
of the $k_2$ parts of size $\Lambda_2$. Similarly, $c_{\gamma}=0$ 
  whenever  $\gamma$ satisfies $\gamma_{i_j^2}\neq \Lambda_2$ for any $1\leq j\leq k_2$.
Continuing, we see that $c_{\gamma}= 0$ for all $\gamma \neq \beta$. That 
is, there is no convex combination (\ref{eqn:convecombfeb16}), as desired.

(III): 
Suppose $\alpha$ is a lattice point in ${\sf Newton}(f)  = {\mathcal P}_{\lambda}\subset {\mathbb R}^{n}.$  
Let $\lambda(\alpha)$ be the rearrangement of $\alpha$ into a partition. 
By symmetry, 
${\mathcal P}_{\lambda(\alpha)} \subseteq {\mathcal 
P}_{\lambda}$.
Then by (\ref{eqn:containDominance}), $\lambda(\alpha) \leq_D \lambda$ 
and so by (\ref{eqn:kostkadominance}), ${\sf K}_{\lambda,\lambda(\alpha)} \neq 0$.  
Since 
$x^\alpha$ appears in $m_{\lambda(\alpha)}(x_1,\ldots,x_n)$, $x^\alpha$ appears in $f(x_1,\ldots,x_n)$ (here we are using the Schur positivity of $f$ and the fact $\ell(\lambda(\alpha))\leq n$).  
Thus $f$ is SNP.
\end{proof}

\begin{Example}[Schur positivity does not imply SNP]\label{exa:SchurposdotimplySNP}
Let
\[f=s_{(8,2,2)}+ s_{(6,6)}.\]
It is enough to show $f(x_1,x_2,x_3)$ is not SNP.
Now, $m_{(8,2,2)}(x_1,x_2,x_3)$ and $m_{(6,6)}(x_1,x_2,x_3)$ 
appear in the monomial expansion of $f(x_1,x_2,x_3)$. 
However, $m_{(7,4,1)}(x_1,x_2,x_3)$ is not in $f(x_1,x_2,x_3)$
since $(7,4,1)$ is not $\leq_D$-comparable with 
$(8,2,2)$ nor $(6,6,0)$. Yet, 
$(7,4,1)=\frac{1}{2}(8,2,2)+\frac{1}{2}(6,6,0)\in {\sf Newton}(f(x_1,x_2,x_3))$. 
Hence $f$ is not SNP.
\qed
\end{Example}

\begin{Example}
The Schur positivity assumption in Proposition~\ref{prop:generalthing}(III) 
is necessary: 
\[f=s_{(3,1)}(x_1,x_2)-s_{(2,2)}(x_1,x_2)=x_1^3x_2+x_1x_2^3\]
is not SNP. \qed
\end{Example}

\begin{Example}
$f\in {\sf Sym}$ can be SNP
without a unique $\leq_D$-maximal term. For example,
\[f=s_{(2,2,2)}+s_{(3,1,1,1)}\] 
is SNP but $(2,2,2)$
and $(3,1,1,1)$ are $\leq_D$-incomparable. An instance of this
from ``nature'' is found in Example~\ref{exa:kron}.\qed
\end{Example}

\begin{Proposition}[Products of Schur polynomials are SNP]
\label{prop:productSchur}
$s_{\lambda^{(1)}}s_{\lambda^{(2)}}\cdots s_{\lambda^{(N)}}\in {\sf Sym}$ is SNP for 
any partitions $\lambda^{(1)},\ldots,\lambda^{(N)}$.
\end{Proposition}
\begin{proof}
We have
\[s_{\lambda}s_{\mu}=\sum_{\nu} {\sf LR}_{\lambda,\mu}^{\nu}s_{\nu}\in {\sf Sym},\]
where ${\sf LR}_{\lambda,\mu}^{\nu}\in {\mathbb Z}_{\geq 0}$ is the {\bf Littlewood-Richardson coefficient}.
By homogeneity, clearly ${\sf LR}_{\lambda,\mu}^{\nu}=0$ unless $|\nu|=|\lambda|+|\mu|$.
Let $\lambda+\mu=(\lambda_1+\mu_1,\lambda_2+\mu_2,\ldots)$. 
It suffices to show $\nu\leq_D \lambda+\mu$ 
whenever ${\sf LR}_{\lambda,\mu}^{\nu}\geq 0$.
Actually, we show $s_{\lambda+\mu}$ is the unique
$\leq_D$-maximal term in the Schur expansion of
$s_{\lambda}e_{\mu'}$. Indeed, since $e_{\mu'}=s_{\mu}+\text{(positive sum of Schur functions)}$, this will suffice.
The strengthening holds by an easy induction on the number of 
nonzero parts of $\mu$ and the Pieri rule (in the form of
\cite[Example~7.15.8]{Stanley}). Alternatively, this is straightforward to prove
from the Littlewood-Richardson rule. Iterating this argument shows that $s_{\lambda^{(1)}}\cdots s_{\lambda^{(N)}}$
has unique $\leq_D$ maximal term $s_{\lambda^{(1)}+\cdots+
\lambda^{(N)}}$ and hence is SNP.
\end{proof}

Let 
\[p_{k}=\sum_{i=1}^n x_i^k\]
be the {\bf power sum symmetric
polynomial}.  Moreover, let 
\[p_{\lambda}:=p_{\lambda_1}p_{\lambda_2}\cdots.\]  

\begin{Proposition}
\label{prop:characterargument}
Let 
\[f=\sum_{\lambda\vdash n} c_{\lambda} p_{\lambda}\in {\sf Sym}\] 
be not identically zero. Assume $c_{\lambda}\geq 0$ for all
$\lambda$, and that $f$ is Schur positive. Then
$f$ is SNP.
\end{Proposition}
\begin{proof} 
Recall, $(n)$ indexes the trivial representation of $S_n$ that sends each $\pi\in 
S_n$ to the $1\times 1$ identity matrix. The character value $\chi^{(n)}(\mu)$,
being the 
trace of this matrix, is 
independent of $\pi$'s conjugacy class $\mu\vdash n$. Hence
$\chi^{(n)}(\mu)=1$. 

We have
\[p_{\mu}=\sum_{\lambda}\chi^{\lambda}(\mu)s_{\lambda}.\]
Therefore,
\[f=\sum_{\lambda\vdash n}c_{\lambda}p_{\lambda}=\left(\sum_{\lambda\vdash n}c_{\lambda}\right)s_{(n)}+\sum_{\lambda\neq (n)}
d_{\lambda}s_{\lambda}.\]
By hypothesis, each $c_{\lambda}\geq 0$. 
Since $f\not\equiv 0$, some $c_{\lambda}>0$ and hence 
$s_{(n)}$ appears. 
Now, $(n)$ is the (unique) maximum in $({\sf Par}(n),\leq_D)$. Also,
since $f$ is Schur positive, each $d_{\lambda}\geq 0$.
Hence the result follows from Proposition~\ref{prop:generalthing}(III).
\end{proof}

Let 
\[\omega:{\sf Sym}\to {\sf Sym}\] 
be the involutive automorphism defined by \[\omega(s_{\lambda})=s_{\lambda'},\]
where $\lambda'$ is the shape obtained by transposing the Young diagram of $\lambda$. 

\begin{Example}[$\omega$ does not preserve SNP]
Example~\ref{exa:SchurposdotimplySNP} shows $f=s_{(8,2,2)}+s_{(6,6)}\in {\sf Sym}$ is not SNP. Now
\[\omega(f)=s_{(3,3,1,1,1,1,1,1)}+s_{(2,2,2,2,2,2)}\in {\sf Sym}.\]
To see that $\omega(f)$ is SNP, it suffices to show that
any partition $\nu$ that is is a linear combination of rearrangements of $\lambda=(3,3,1,1,1,1,1,1)$ and $\mu=(2,2,2,2,2,2)$ satisfies $\nu\leq_D \lambda$ or $\nu\leq_D\mu$. We leave the details to the reader.
\qed
\end{Example}

\subsection{Examples and counterexamples}

\

\begin{Example}[Monomial symmetric and forgotten symmetric polynomials]
It is immediate from 
(\ref{eqn:newtonmonomial}) and
(\ref{eqn:containDominance}) that
\[m_{\lambda}\in {\sf Sym} \text{ \ is SNP $\iff\lambda=(1^n)$}.\]

The {\bf forgotten symmetric functions} are defined by 
\[f_{\lambda}=\varepsilon_{\lambda}\omega(m_{\lambda})\]
where $\varepsilon_{\lambda}=(-1)^{|\lambda|-\ell(\lambda)}$.
\end{Example}

\begin{Proposition}
$f_{\lambda}\in {\sf Sym}$ is SNP if and only if $\lambda=(1^n)$.
\end{Proposition}
\begin{proof}
($\Leftarrow$) If $\lambda=(1^n)$ then $m_{\lambda}=s_{(1^n)}$ and 
$f_\lambda=s_{(n,0,0,\ldots,0)}$ which is SNP.

($\Rightarrow$) We use the following formula \cite[Exercise~7.9]{Stanley}:
\[f_{\lambda}=\sum_{\mu} a_{\lambda\mu} m_{\mu}\]
where $a_{\lambda\mu}$ is the number of distinct rearrangements $(\gamma_1,\ldots,\gamma_{\ell(\lambda)})$ of
$\lambda=(\lambda_1,\ldots,\lambda_{\ell(\lambda)})$ such that
\begin{equation}
\label{eqn:feb20abcd}
\left\{\sum_{s=1}^{i}\gamma_s:1\leq i\leq \ell(\lambda)\right\}\supseteq \left\{\sum_{t=1}^j \mu_t: 1\leq j\leq \ell(\mu)\right\}.
\end{equation}

Suppose $\lambda\neq (1^n)$. If $\mu=(1^n)$ then  
\[\left\{\sum_{t=1}^j \mu_t: 1\leq j\leq \ell(\mu)\right\}=\{1,2,3,\ldots,n\}.\]
On the other hand, $\ell(\lambda)<n$ and hence
the set on the lefthand side of (\ref{eqn:feb20abcd}) has size strictly smaller than $n$.
Thus $a_{\lambda,(1^n)}=0$.

Now, $(1^n)\in {\mathcal P}_{\mu}$ 
for all $\mu=(\mu_1,\ldots,\mu_n)\vdash n$. Then since 
$f_{\lambda}$ is $m$-positive, $(1^n)\in {\sf Newton}(f_{\lambda})$ so long as we are working
in at least $\deg(f_{\lambda})$ many variables.
If $\lambda\neq (1^n)$ then $a_{\lambda,(1^n)}=0$ means $(1^n)$
is not an exponent vector of $f_{\lambda}$. This proves the contrapositive of
($\Rightarrow$).
\end{proof}

\begin{Example}[Elementary and complete homogeneous symmetric polynomials]
The {\bf elementary symmetric polynomial} is defined by
\[e_k(x_1,\ldots,x_n)=\sum_{1\leq i_1<i_2<\ldots<i_k\leq n} x_{i_1}x_{i_2}\cdots x_{i_k}.\]
Also define 
\[e_{\lambda}=e_{\lambda_1}e_{\lambda_2}\cdots\]
The {\bf complete homogeneous symmetric polynomial} $h_k(x_1,\ldots,x_n)$ is the sum of all degree $k$ 
monomials. Also define
\[h_{\lambda}=h_{\lambda_1}h_{\lambda_2}\cdots.\]
\end{Example}

\begin{Proposition}
\label{cor:eandhSNP}
Each $e_{\lambda}$ and $h_{\lambda}$ is SNP.
\end{Proposition}
\begin{proof}
Since $e_k=s_{(1^k)}$ and $h_k=s_{(k)}$ the claim holds by Proposition~\ref{prop:productSchur}.
\end{proof}

The {\bf Minkowski sum} of two polytopes ${\mathcal P}$ and 
${\mathcal Q}$ is  
\[{\mathcal P}+{\mathcal Q}=\{p+q:p\in {\mathcal P},q\in{\mathcal Q}\}.\]
Thus,
\[{\sf Newton}(f\cdot g)={\sf Newton}(f)+{\sf Newton}(g).\]
By the Pieri rule,
\[e_{(1)}(x_1,\ldots,x_n)^k=\sum_{\lambda} {\sf f}^{\lambda} s_{\lambda}(x_1,\ldots,x_n).\]
In particular $s_{(k,0,\ldots,0)}$ appears on the right hand side.
Since $\lambda\leq_D (k)$ for all $\lambda\vdash k$, by 
Proposition~\ref{prop:generalthing}(I) one recovers that the Minkowski sum of $k$ 
regular simplices 
in ${\mathbb R}^n$ is ${\mathcal P}_{(k,0,\ldots,0)}$. Similarly, by the argument of
Proposition~\ref{prop:productSchur}, 
${\mathcal P}_{\lambda}={\sf Newton}(e_{\lambda'})$ and hence one recovers that ${\mathcal P}_{\lambda}$ is a 
Minkowski sum of hypersimplices. For earlier work see, e.g., \cite{Postnikov, Croitoru, Morris}.\qed

\begin{Example}[$e$-positivity does not imply SNP]
$f\in {\sf Sym}$ is {\bf $e$-positive} if 
$f=\sum_{\lambda}a_{\lambda} e_{\lambda}$ where $a_{\lambda}\geq 0$ for every $\lambda$. 
(Since 
\[e_{\lambda}=\sum_{\mu}{\sf K}_{\mu',\lambda}s_{\mu},\] 
$e$-positivity implies Schur positivity.) Look at
\[f=e_{(3,3,1,1,1,1,1,1)}+ e_{(2,2,2,2,2,2)} \in {\sf Sym}.\] 
In the monomial expansion, $m_{(8,2,2)}$ and $m_{(6,6)}$ appear. However, $m_{(7,4,1)}$ does not appear. This implies $f$ is not SNP.
\qed
\end{Example}

\begin{Example}[More on power sum symmetric polynomials]
Recall the power sum symmetric polynomials defined immediately before
Proposition~\ref{prop:characterargument}. Clearly $p_k$ is not SNP
if $k>1$ and $n>1$. Also,
$p_{\lambda}$ is not SNP for $n>1$ whenever 
$\lambda_i\geq 2$ for all $i$. This is since $x_1^{|\lambda|}$ and 
$x_2^{|\lambda|}$ both appear as monomials in $p_{\lambda}$
but $x_1^{|\lambda|-1}x_2$ does not. Furthermore:
\end{Example}

\begin{Proposition}
$p_\lambda \in {\sf Sym}_n$ for $n > \ell(\lambda)$ is SNP if and only if $\lambda = (1^k)$.
\end{Proposition}
\begin{proof}
($\Leftarrow$) If $\lambda = (1^k)$, $p_\lambda = e_\lambda$ which is SNP by Proposition~\ref{cor:eandhSNP}.

($\Rightarrow$) Suppose $\lambda_1 \geq 2$ and let $\ell = \ell(\lambda)$.  Then since 
$n > \ell$, $x_1^{\lambda_1}x_2^{\lambda_2}\cdots x_\ell^{\lambda_\ell}$ and 
$x_2^{\lambda_2}\cdots x_\ell^{\lambda_\ell}x_{\ell + 1}^{\lambda_1}$ are monomials in 
$p_\lambda$. Thus, 
\[ (\lambda_1-1,\lambda_2,\ldots,\lambda_\ell,1) = 
\frac{\lambda_1-1}{\lambda_1}(\lambda_1,\lambda_2,\ldots,\lambda_\ell,0) + 
\frac{1}{\lambda_1}(0,\lambda_2,\ldots,\lambda_\ell,\lambda_1) \in {\sf 
Newton}(p_\lambda).\] However, this point cannot be an exponent vector since it has 
$\ell+1$ nonzero components whereas every monomial of $p_\lambda$ uses at most $\ell$ 
distinct variables.
\end{proof}

\begin{Example}[The resultant, the Gale-Ryser theorem and $(0,1)$-matrices]
Let 
\[f=\sum_{i=0}^m a_i z^i \text{\  and 
$g=\sum_{i=1}^n b_i z^i$}\] 
be two polynomials of degree $m$ and $n$
respectively and with roots $\{x_1,\ldots,x_m\}$
and $\{y_1,\ldots,y_n\}$ respectively (not necessarily distinct). The
{\bf resultant} is
\[R(f,g)=a_m^n b_n^m\prod_{i=1}^m \prod_{j=1}^n (x_i-y_j).\]
This polynomial is separately symmetric in the $x$ and $y$ variables.
In \cite{GKZ} the Newton polytope of $R(f,g)$ is determined; see also the book
 \cite{GKZbook}. However, we are not aware of the following result appearing
explicitly in the literature:

\end{Example}

\begin{Theorem} 
\label{prop:resultantSNP}
$R(f,g)$ is SNP.
\end{Theorem}
\begin{proof}
Consider
\[F=\prod_{i=1}^m\prod_{j=1}^n(1+x_i y_j).\]
In fact, $[x^{\alpha}y^{\beta}]F$ equals the number of $(0,1)$-matrices
of dimension $m\times n$ whose row sums are given by $\alpha$ and column
sums are given by $\beta$; see, e.g., \cite[Proposition~7.4.3]{Stanley}. Let ${\sf M}(\alpha,\beta)$ equal the number of these matrices.
The Gale-Ryser theorem states
\begin{equation}
\label{eqn:GR}
{\sf M}(\alpha,\beta)>0 \iff \lambda(\beta)\leq_D \lambda(\alpha)',
\end{equation}
where $\lambda(\gamma)$ is the partition obtained by sorting a nonnegative integer sequence in decreasing order.
Call a pair of vectors $(\alpha,\beta)\in {\mathbb Z}_{\geq 0}^{m+n}$ a {\bf GR pair} if it satisfies either of the equivalent
conditions in (\ref{eqn:GR}).

In fact $F$ is SNP. Suppose 
\[(\alpha^{(1)},\beta^{(1)}), (\alpha^{(2)},\beta^{(2)}),\ldots, (\alpha^{(N)},\beta^{(N)})\] 
are GR pairs and
\[(\alpha,\beta)=\sum_{t=1}^N d_i (\alpha^{(t)},\beta^{(t)})\]
with $d_i\geq 0$ and $\sum_{t=1}^N d_i=1$ be a convex combination. The SNPness of $F$
is equivalent to the claim $(\alpha,\beta)$ is a GR pair whenever $(\alpha,\beta)\in {\mathbb Z}_{\geq 0}^{m+n}$.
The latter claim is immediate from \cite[Theorem~3, part~1]{Barvinok} which establishes the ``approximate
log-concavity'' of ${\sf M}(\alpha,\beta)$. We thank A.~Barvinok for pointing out this reference to us.

Now notice that
\begin{align*}
\text{$F$ is SNP} & \text{$\iff \prod_{i=1}^m\prod_{j=1}^n(1+x_i y_j^{-1})$ is SNP}\\
& \text{$\iff y_1^m y_2^m\cdots y_n^m \prod_{i=1}^m\prod_{j=1}^n(1+x_i y_j^{-1})$ is SNP}\\
& \text{$\iff \prod_{i=1}^m\prod_{j=1}^n (x_i+y_j)$ is SNP}\\
& \text{$\iff \prod_{i=1}^m\prod_{j=1}^n (x_i-y_j)$ is SNP}\\
& \text{$\iff R(f,g)$ is SNP.}
\end{align*}
The final equivalence is true since $a_0,b_0\neq 0$ and the previous equivalence holds since the polynomials
in the third and fourth lines clearly share the same monomials. This relation between $F$ and $R(f,g)$ appears
in \cite{GKZ} where the authors use it to obtain a formula for the monomials of $R(f,g)$ in terms of counts for
$(0,1)$-matrices.
\end{proof}

Conjecture~\ref{conj:double} claims a generalization of Theorem~\ref{prop:resultantSNP}; see 
Example~\ref{exa:doubleschubres}.\qed
  
\begin{Example}[Powers of the Vandermonde]
The {\bf Vandermonde determinant} is
\[a_{\delta_n}=\prod_{1\leq i<j\leq n} (x_i-x_j).\]
(This polynomial is only skew-symmetric.) 
It is known that 
\[{\sf Newton}(a_{\delta_n})={\mathcal P}_{(n-1,n-2,\ldots,2,1,0)}\subset
{\mathbb R}^n;\] 
see e.g., \cite[Proposition~2.3]{Postnikov}. 
\end{Example}

\begin{Proposition}
\label{prop:vandy}
$a_{\delta_n}$ is SNP if and only if $n \leq 2$.
\end{Proposition}

The classical {\bf discriminant} is
$\Delta_n=\alpha_{\delta_n}^2$. 
Its Newton polytope was also determined by
in work of I.~M.~Gelfand-M.~Kapranov-A.~V.~Zelevinsky \cite{GKZ}.

\begin{Proposition}
\label{abel}
$\Delta_n$ is SNP if and only if $n\leq 4$.
\end{Proposition}

Proposition~\ref{abel} is a curious coincidence with the Abel-Ruffini theorem.

Our proofs of Propositions~\ref{prop:vandy} and~\ref{abel} will use this lemma:

\begin{Lemma}
\label{prop:higher}
If $a_{\delta_n}^k$ is not SNP, then $a_{\delta_{n+1}}^k$ is not SNP.
\end{Lemma}
\begin{proof}
Suppose $a_{\delta_n}^k$ is not SNP. There exists a lattice point
$\alpha \in {\sf Newton}(a_{\delta_n}^k)$ that is not an 
exponent vector of $a_{\delta_n}^k$. Hence we have a convex
combination 
\[\alpha=\sum_{i=1}^N c_i \beta^i\]
where $\beta^i$ is an exponent vector. For $\gamma\in {\mathbb R}^n$,
let $\gamma'=(\gamma,kn)\in {\mathbb R}^{n+1}$.  Since 
\[a_{\delta_{n+1}}^k=a_{\delta_n}^k \times \prod_{i=1}^n (x_i-x_{n+1})^k,\]
each $(\beta^i)'$ is an exponent vector of $a_{\delta_{n+1}}^k$ and hence
$\alpha'$ is a lattice point of ${\sf Newton}(a_{\delta_{n+1}}^k)$. Since $x^{\alpha'} = x^\alpha x_{n+1}^{kn}$ and 
$x_{n+1}$ does not appear in $a_{\delta_{n}}^k$, if 
$\alpha'$ is an exponent vector of $a_{\delta_{n+1}}^k$, $\alpha$ is an exponent vector of $a_{\delta_n}^k$, a contradiction.  Thus $a_{\delta_{n+1}}^k$ is not SNP.
\end{proof}

\noindent\emph{Proof of Propositions~\ref{prop:vandy}
and~\ref{abel}:}
Clearly, $a_{\delta_n}$ is SNP for $n =1,2$. One checks that $(1,1,1) \in {\sf Newton}(a_{\delta_3})$ but is not an actual exponent vector.

Separately, one checks $\Delta_n$ is SNP for $n\leq 4$.
Also $\Delta_5$ is not SNP. In fact, the only lattice points that
are not exponent vectors are all $5!$ rearrangements of $(1,3,4,5,7)$. 

Now apply Lemma~\ref{prop:higher} to complete an induction argument
for each of the two propositions being proved.\qed

\begin{Conjecture}
For all $k$, there exists $N_k$ such that $a_{\delta_n}^k$ is not SNP for any $n \geq N_k$. 
\end{Conjecture}

More precisely, for $1\leq j\leq 4$ we computed $N_{2j-1}=3$ and moreover that
$(1,3j-2,3j-2)$ is a lattice point that is not an exponent vector. Moreover,
$N_2=5, N_4=4,N_6=4,N_8=3$. For more on (higher) powers of the Vandermonde, see, e.g., \cite{Scharf, Ball}.\qed

\begin{Example}[$q$-discriminant]
The {\bf $q$-discriminant} is $\prod_{1\leq i<j\leq n}(x_i-qx_j)$. At $q=-1$, 
\[f_n=\prod_{1\leq i<j\leq n}(x_i+x_j)\in {\sf Sym}_n.\]
It is known that 
\[f_n=s_{\rho_n}(x_1,x_2,\ldots,x_n)  \text{ \ \ where $\rho_n=(n-1,n-2,\ldots,3,2,1,0)$}.\] 
Hence $f_n$ is SNP and ${\sf Newton}(f_n)={\mathcal P}_{\rho_n}\subset {\mathbb R}^n$.
\qed
\end{Example}

\begin{Example}[Totally nonnegative matrices]
Let 
\[M=(m_{ij})_{1\leq i,j\leq n}\] be an $n\times n$ {\bf totally nonnegative real matrix}. That is, every determinant of a square
submatrix is nonnegative. Define
\[F_M=\sum_{w\in S_n} \left(\prod_{i=1}^{n} m_{i,w(i)}\right) p_{\lambda(w)},\]
where $\lambda(w)$ is the cycle type of $w$. 
\end{Example}

\begin{Theorem}
\label{thm:Stem}
$F_M$ is SNP.
\end{Theorem}
\begin{proof}
By assumption, $m_{ij}\geq 0$. A theorem of J.~R.~Stembridge \cite{Stembridge} (cf.~\cite[Exercise~7.92]{Stanley}) states that
$F_{M}$ is also Schur positive. Now apply Proposition~\ref{prop:characterargument}.
\end{proof}

\begin{Example}[Redfield--P\'{o}lya theory]
Let $G$ be a subgroup of $S_n$. The {\bf cycle index polynomial} is 
\[Z_G=\frac{1}{|G|}\sum_{g\in G} p_{\lambda(g)},\]
where $\lambda(g)$ is the cycle type of $g$.
\end{Example}

\begin{Theorem}
\label{thm:cycleindex}
$Z_G$ has SNP.\qed
\end{Theorem}
\begin{proof}
 It is true that
\[Z_G=\sum_{\lambda} c_{\lambda} s_{\lambda},\]
where each $c_{\lambda}\in{\mathbb Z}_{\geq 0}$; see \cite[pg. 396]{Stanley}: this positivity is known for
representation-theoretic reasons (no combinatorial proof is available). Now use 
Proposition~\ref{prop:characterargument}.
\end{proof}

\begin{Example}[C.~Reutenauer's $q_{\lambda}$ basis]
C.~Reutenauer \cite{Reutenauer} introduced a new basis $\{q_{\lambda}\}$ of symmetric polynomials, recursively defined by
setting
\[\sum_{\lambda\vdash n} q_{\lambda} = s_{(n)},\]
where 
$q_{\lambda}=q_{\lambda_1}q_{\lambda_2}\cdots$. 
\end{Example}

\begin{Theorem}
\label{thm:Reutenauer}
$q_{\lambda}$ has SNP.
\end{Theorem}
\begin{proof}
Reutenauer in \emph{loc. cit.} conjectured that $-q_{(n)}$ is Schur positive for $n\geq 2$.
Indeed,  
\[q_{(1)}=s_{(1)}, q_{(2)}=-s_{(1,1)}, q_{(3)}=-s_{(2,1)}.\] 
Reutenauer's conjecture was
later established by W.~M.~Doran IV \cite{Doran}. The proof sets
\[f(n,k)=\sum_{\lambda\vdash n, {\rm min}(\lambda_i)\geq k} q_{\lambda}.\]
The argument inducts on $n$ and proceeds by showing that 
\[-f(n,k)=s_{(n-1,1)}+\sum_{2\leq i<k} (-f(i,i))(-f(n-i,i)).\]
His induction claim is that $-f(n,k)$ is Schur positive for $k\geq 2$. Let us 
strengthen his induction hypothesis, and
assume $-f(n,k)$ is Schur positive with $s_{(n-1,1)}$ as the unique
$\leq_D$ maximal term. In the induction step, note each 
$s_{\alpha}$ appearing in $-f(i,i)$ has $\alpha_1\leq i-1$ and each
$s_{\beta}$ in $-f(n-i,i)$ has $\beta_1\leq n-i-1$. Thus, 
by the argument of Proposition~\ref{prop:productSchur},
if $s_{\gamma}$ appears in
$s_{\alpha}s_{\beta}$ then $\gamma_1\leq n-2$, implying the strengthening we need.

It follows from the above argument and  the Littlewood-Richardson rule 
that if $\lambda=(\lambda_1, \ldots, \lambda_{\ell},1^r)$ where each $\lambda_{i}\geq 2$ then $q_{\lambda}$ has a unique $\leq_D$-leading term $s_{a,b}$ where $a = |\lambda |-\ell$ and $b = \ell$. Thus, $q_{\lambda}$ has SNP by 
Proposition~\ref{prop:generalthing}(III).
\end{proof}

\begin{Example}[Stanley's chromatic symmetric polynomial]
For a graph $G$, let $c_{G}(x_1,\ldots,x_n)$ be \emph{Stanley's chromatic symmetric polynomial} \cite{Stanley1995}. If $G=K_{1,3}$,
\[c_{G}(x_1,x_2,\ldots)=
x_1^3x_2+x_1 x_2^3+\cdots\] 
is not SNP. \qed
\end{Example}

\begin{Example}[Kronecker product of Schur polynomials]
\label{exa:kron}
The {\bf Kronecker product} is
\[s_{\lambda}*s_{\mu}=\sum_{\nu} {\sf Kron}_{\lambda,\mu}^{\nu} s_{\nu}\in {\sf Sym}.\]
${\sf Kron}_{\lambda,\mu}^{\nu}$ is the {\bf Kronecker coefficient}, 
the multiplicity of the $S_n$-character $\chi^{\nu}$ appearing in $\chi^\lambda\otimes\chi^\mu$. 
We conjecture that $s_{\lambda}*s_{\mu}$ is SNP. We have verified this for
all $\lambda,\mu\in {\sf Par}(n)$ for $1\leq n\leq 7$. Consider 
\[s_{(4,2)}*s_{(2,2,1,1)} 
=s_{(1, 1, 1, 1, 1, 1)} + s_{(2, 1, 1, 1, 1)} + 2s_{(2, 2, 1, 1)} + s_{(3, 1, 1, 1)} 
+ 2s_{(3, 2, 1)} + s_{(3, 3)} + s_{(4, 1, 1)}.\]
Notice $(3,3)$ and $(4,1,1)$ are both $\leq_D$-maximal. Hence in this case, 
SNPness cannot be blamed on Proposition~\ref{prop:generalthing}(III); cf.~\cite[Lemma~3.2]{Vallejo2010} and \cite{Vallejo2000}.
\qed
\end{Example}
\begin{Example}[Lascoux-Leclerc-Thibon (LLT) polynomials]
A.~Lascoux-B.~Leclerc-J.~Y.~Thibon \cite{LLT} introduced 
$G_{\lambda}^{(m)}(X;q)$.
$G_{\lambda}^{(m)}(X;1)$ is a product of Schur polynomials.
Hence $G_{\lambda}^{(m)}(X;1)$ is SNP by Proposition~\ref{prop:productSchur}. 
$G_{\lambda}^{(m)}(X;q)\in {\sf Sym}_n[q]$ is not always SNP. 
One example is 
\[G_{(3,3)}^{(2)}(x_1,x_2;q) = q^3(x_1^3+x_1^2 x_2+x_1 x_2^2
+x_2^3)+q(x_1^2 x_2 +x_2^2 x_1).\]
LLT polynomials arise in the study of Macdonald polynomials, the topic of 
Section~\ref{sec:mac}. (Another related topic
is \emph{affine Schubert calculus}; see the book \cite{kschurbook}.)
\qed
\end{Example}

\section{Macdonald polynomials}
\label{sec:mac}

\subsection{Symmetric and nonsymmetric Macdonald polynomials}
In \cite{Macdonald:newclass}, I.~G.~Macdonald introduced
a family of polynomials depending of parameters $q$ and $t$. Define an inner product 
$\langle\bullet,\bullet\rangle_{q,t}$ on {\sf Sym} by
\[\langle p_{\lambda},p_{\mu}\rangle_{q,t}=\delta_{\lambda,\mu}z_{\lambda}(q,t),\]
where
\[z_{\lambda}(q,t)=z_{\lambda}\prod_{i=1}^{\ell(\lambda)}\frac{1-q^{\lambda}}{1-t^{\lambda}},\]
and 
\[z_{\lambda}=\prod_{r\geq 1}r^{m_r}m_r!\] 
for $\lambda=(1^{m_1}2^{m_2}\cdots)$.
{\bf Macdonald polynomials}
$\{P_{\mu}(X;q,t)\}$ are uniquely determined by 
\begin{equation}
\label{eqn:Macdef1}
P_{\lambda}(X;q,t)=m_{\lambda}(X)+\sum_{\mu<_D\lambda} {\sf c}_{\lambda,\mu}(q,t)m_{\mu}(X)
\end{equation}
where ${\sf c}_{\lambda,\mu}(q,t)\in {\mathbb Q}(q,t)$, together with
\[\langle P_{\lambda},P_{\mu}\rangle_{q,t}=0 \text{\ \ if $\lambda\neq \mu$}.\] 

\begin{Theorem}
\label{prop:PisSNP}
$P_{\lambda}(X;q=q_0,t=t_0)$ is SNP, and ${\sf Newton}(P_{\lambda}(x_1,\ldots,x_n;q=q_0,t=t_0))={\mathcal P}_{\lambda}\subset {\mathbb R}^n$ whenever $n\geq \ell(\lambda)$,
for any $(q_0,t_0)$ in a Zariski open subset of ${\mathbb C}^2$.
\end{Theorem}

\begin{Lemma} 
\label{lemma:anyeval}
${\sf Newton}(P_{\lambda}(x_1,\ldots,x_n;q=q_0,t=t_0))={\mathcal P}_{\lambda}\subset {\mathbb R}^n$ whenever $n\geq \ell(\lambda)$,  
for any $(q_0,t_0)\in {\mathbb C}^2$. 
\end{Lemma}
\begin{proof}
This is by (\ref{eqn:Macdef1}) and Proposition~\ref{prop:generalthing}(I).
Since $n\geq \ell(\lambda)$, $s_{\lambda}(x_1,\ldots,x_n)
\not\equiv 0$.
\end{proof}

\begin{Lemma}
\label{lemma:maceval}
Fix $q_0,t_0\in {\mathbb C}$.
$P_{\lambda}(X;q=q_0,t=t_0)$ is SNP if and only if ${\sf c}_{\lambda,\mu}(q_0,t_0)\neq 
0$ for all $\mu<_D\lambda$.
\end{Lemma}
\begin{proof}
($\Rightarrow$) By Lemma~\ref{lemma:anyeval}, 
${\sf Newton}(P_{\lambda}(X;q=q_0,t=t_0))={\mathcal P}_{\lambda}$.
Thus each $\mu<_D\lambda$ appears as a lattice
point of ${\sf Newton}(P_{\lambda}(X;q=q_0,t=t_0))$. 
Since we assume $P_{\lambda}(X;q=q_0,t=t_0)$ is SNP, 
$[x^{\mu}]P_{\lambda}(X;q=q_0,t=t_0)\neq 0$. 
Among monomial symmetric functions, $x^\mu$ only appears in $m_{\mu}$. 
Hence ${\sf c}_{\lambda,\mu}(q_0,t_0)\neq 0$, as desired. 

The proof of $(\Leftarrow)$ just reverses the above argument, using the
fact that 
\[\mu\in{\sf Newton}(P_{\lambda}(X;q=q_0,t=t_0))\iff
\alpha\in {\sf Newton}(P_{\lambda}(X;q=q_0,t=t_0))\]
for any rearrangement $\alpha$ of $\mu\in {\mathbb R}^n$.
\end{proof}

\noindent
\emph{Proof of Theorem~\ref{prop:PisSNP}:}
The Newton polytope assertion is by Lemma~\ref{lemma:anyeval}.
Now 
\[P_{\lambda}(X;0,0)=s_{\lambda}(X)\] 
and $m_{\mu}$
appears in $s_{\lambda}$ for every $\mu<_D\lambda$. 
Hence ${\sf c}_{\lambda,\mu}(q,t)\not\equiv 0$.
Now choose $q,t$ that is neither a pole nor a root
of any of these rational functions (for $\mu<_D\lambda$).  Therefore the SNP assertion follows from Lemma~\ref{lemma:maceval}.\qed

The {\bf Hall-Littlewood polynomial} is $P_{\lambda}(X;t):=P_{\lambda}(X,q=0,t)$. One has
\[P_{\lambda}(X;t)=\sum_{\mu} {\sf K}_{\lambda,\mu}(t)s_{\mu}(X)\]
where ${\sf K}_{\lambda,\mu}(t)$ is the {\bf Kostka-Foulkes polynomial}. It is known that
\[{\sf K}_{\lambda,\mu}(t)=\sum_{T}t^{{\rm charge}(T)}.\]
The sum is over all semistandard tableau $T$ of shape $\lambda$ and content $\mu$. It is true that ${\rm charge}(T)\in {\mathbb Z}_{\geq 0}$. 
Since these tableaux can only occur
if $\mu\leq_D \lambda$, ${\sf K}_{\lambda,\mu}(t)\not\equiv 0$ 
if and only if $\mu\leq_D\lambda$. Hence we immediately obtain:
\begin{Proposition}
\label{prop:HallisSNP}
If $t_0>0$ then $P_{\lambda}(X;t=t_0)\in {\sf Sym}$ is SNP and 
${\sf Newton}(P_{\lambda}(x_1,\ldots,x_n;t=t_0)={\mathcal P}_{\lambda}\subset {\mathbb R}^n$ whenever $n\geq \ell(\lambda)$.
\end{Proposition}

The {\bf Schur $P-$ polynomial} is
\[SP_{\lambda}(X)=\sum_{T} x^{T}.\]
The sum is over shifted semistandard Young tableaux of
a partition $\lambda$ with distinct parts.
There is also the {\bf Schur $Q-$ polynomial}, 
\[SQ_{\lambda}(X)=2^{\ell(\lambda)} SP_{\lambda}.\]

\begin{Proposition}
$SP_{\lambda}(X)$ and $SQ_{\lambda}(S)$  are SNP and 
\[{\sf Newton}(SP_{\lambda}(X))={\sf Newton}(SQ_{\lambda}(X))={\mathcal P}_{\lambda}.\]
\end{Proposition}
\begin{proof}
In fact,
\[SP_{\lambda}(X)=P_{\lambda}(X;t=-1);\]
see \cite{Stembridge1989}. Also $K_{\lambda,\lambda}(t)=1$. 
Now, $SP_{\lambda}$ is Schur-positive; see, e.g., 
\cite[p.~131--132]{Stembridge1989}. Thus the result follows from
Proposition~\ref{prop:generalthing}(III).
\end{proof}

The \emph{modified Macdonald polynomial} 
${\widetilde H}_{\lambda}(X;q,t)$ is a certain transformation of
$P_{\lambda}(X;q,t)$ also introduced in
\cite{Macdonald:newclass}.

\begin{Proposition}
\label{prop:modifiedisSNP}
For any $q_0,t_0> 0$, ${\widetilde H}_{\lambda}(X;q=q_0,t=t_0)$ is SNP and 
${\sf Newton}({\widetilde H}_{\lambda}(x_1,\ldots,x_n;q=q_0,t=t_0))={\mathcal P}_{|\lambda|}\subset {\mathbb R}^n$ whenever $n\geq |\lambda|$.
\end{Proposition}
\begin{proof}
A formula of J.~Hagland-M.~Haiman-N.~Loehr \cite{HHL:symm} states that
\[{\widetilde H}_{\lambda}(X;q,t)=\sum_{\sigma:\lambda\to {\mathbb Z}_{>0}} x^{\sigma}q^{{\rm inv}(\sigma)}t^{{\rm maj}(\sigma)},\]
where $\sigma$ is any assignment of positive integers to the boxes of
$\lambda$. Also, ${\rm inv}(\sigma)$ and ${\rm maj}(\sigma)$ are certain combinatorially defined
statistics, whose specifics we do not need here. Thus, for $q,t>0$, every monomial of degree $|\lambda|$ appears.
\end{proof}

However, ${\widetilde H}_{(3,1,1)}(x_1,x_2,x_3,x_4,x_5;q,t)$ is not SNP as it misses the monomial $qtx_3x_4^4$.

\begin{Example}[modified $q,t$-Kostka polynomials are not SNP]
Consider the expansion
\[{\widetilde H}_{\lambda}(X;q,t)=\sum_{\mu} {\widetilde {\sf K}}_{\lambda,\mu}(q,t)s_{\mu}(X).\]
The coefficients ${\widetilde {\sf K}}_{\lambda,\mu}(q,t)$ are the {\bf (modified) $q,t$-Kostka coefficients}.
Now,
\[{\widetilde {\sf K}}_{(2,2,2),(2,1,1,1)}(q,t)=qt^7+t^8+qt^5+t^6+qt^4.\] 
Hence, ${\widetilde {\sf K}}_{\lambda,\mu}(q,t)$ need not be SNP.\qed
\end{Example}

Let $\alpha\in {\mathbb Z}_{\geq 0}^n$.
There is the {\bf nonsymmetric Macdonald polynomial}
$E_{\alpha}(x_1,\ldots,x_n;q,t)$; see \cite{HHL:nonsymm} for details. 

It is part of a definition that 
\[E_{\alpha}(X;q,t)=x^{\alpha}+
\sum_{\beta<_{S}\alpha} {\sf d}_{\alpha,\beta}(q,t)x^\beta\] 
where
${\sf d}_{\alpha,\beta}(q,t)\in {\mathbb Q}(q,t)$. S.~Sahi
\cite{Sahi} proved each ${\sf d}_{\alpha,\beta}(q,t)\not\equiv 0$. Here
$<_{S}$ is the ordering whose covering relations are that if 
$\alpha_i<\alpha_j$
then $t_{ij}(\alpha)<_S \alpha$ (where $t_{ij}(\alpha)$ swaps positions $i$ and $j$ of $\alpha$). If also $\alpha_j-\alpha_i>1$ then 
$\alpha+e_i-e_j<_S t_{ij}(\alpha)$; see
\cite[Section~2.1]{HHL:nonsymm}. Let ${\widehat {\mathcal P}_{\alpha}}$
be the convex hull of all $\beta\in {\mathbb Z}_{\geq 0}^n$ such that
$\beta\leq_S \alpha$.
Thus ${\widehat {\mathcal P}_{\alpha}}$ is the Newton polytope of 
$E_{\alpha}(X;q=q_0,t=t_0)$ for any generic choice of 
$(q_0,t_0)\in {\mathbb C}^2$. 
The conjecture below says $E_{\alpha}(X;q,t)$ is ``generically SNP'':

\begin{Conjecture}
\label{conj:genericnonsymm}
If $\beta\in {\widehat{\mathcal P}_{\alpha}}$ and $\beta\in {\mathbb Z}_{\geq 0}^n$ then $\beta\leq_S \alpha$.
\end{Conjecture} 
Conjecture~\ref{conj:genericnonsymm} has been checked for $n\leq 7$ and whenever $|\alpha|\leq 7$.

\subsection{Keys and Demazure atoms}
\label{subsection:keysandatoms}
Complementing the above analysis, 
we now investigate SNP for two specializations of $E_{\alpha}(X;q,t)$.
The first is $\kappa_{\alpha}=E_{\alpha}(X;q=\infty,t=\infty)$
\cite[Theorem~3]{Ion}. The {\bf Demazure operator} is
\[\pi_i(f)=\partial_i(x_i \cdot f ), \mbox{ \ for $f\in {\mathbb Z}[x_1,x_2,\ldots]$.}\]

Let $\alpha=(\alpha_1,\alpha_2,\ldots)\in {\mathbb Z}_{\geq 0}^\infty$ and suppose that
$|\alpha|=\sum_i \alpha_i<\infty$. Define the {\bf key polynomial} $\kappa_{\alpha}$ to be
\[x^{\alpha}:=x_1^{\alpha_1}x_2^{\alpha_2}\cdots, \mbox{\ \ \ if $\alpha$ is weakly decreasing.}\]
Otherwise, set
\begin{equation}
\label{eqn:keypidef}
\kappa_{\alpha}=\pi_i(\kappa_{\widehat \alpha}) \mbox{\ where $\widehat\alpha=(\ldots,\alpha_{i+1},\alpha_i,\ldots)$ and
$\alpha_{i+1}>\alpha_{i}$.}
\end{equation}
The key polynomials
form a ${\mathbb Z}$-basis of ${\mathbb Z}[x_1,x_2,\ldots]$; see work of V.~Reiner--M.~Shimozono \cite{Reiner.Shimozono} (and references therein) for more on $\kappa_{\alpha}$.

Define ${\sf D}_{\alpha}$ to be the ``skyline'' diagram with a left-justified row of $\alpha_i$ boxes in row $i$. 

\begin{Conjecture}
\label{conj:keytopeineq}
${\mathcal S}_{{\sf D}_{\alpha}}={\sf Newton}(\kappa_{\alpha})$.
\end{Conjecture}

We have a proof (omitted here) of the ``$\supseteq$'' part of Conjecture~\ref{conj:keytopeineq}. See Remark~\ref{remark:keytopehalftrue}.

\begin{Conjecture}
\label{conj:kappaSNP}
$\kappa_{\alpha}$ has SNP.
\end{Conjecture}

We have a second conjectural description of ${\sf Newton}(\kappa_{\alpha})$.
Let 
\[m_{ij}(\alpha) = \alpha + e_i - e_j.\]  
Then for any composition $\alpha$, let $\beta <_{\kappa} \alpha$ if $\beta$ can be generated from $\alpha$ by applying a sequence of the moves $t_{ij}$ for $\alpha_i < \alpha_j$,  and $m_{ij}$ if $\alpha_j - \alpha_i > 1$. 

\begin{Conjecture}
\label{conj:keyBruhatCondition}
$\kappa_\alpha = x^\alpha + \sum_{\beta <_{\kappa} \alpha} {\sf Key}_{\alpha,\beta} x^\beta$ with ${\sf Key}_{\alpha,\beta} > 0$ for all $\beta <_{\kappa} \alpha$.
\end{Conjecture}

(Observe that $\beta <_{\kappa} \alpha$, then $\beta <_S \alpha$.
However, the converse fails as $11 <_S 20$ but one does not have
$11 <_{\kappa} 20$.)

For two compositions $\gamma$ and $\alpha$ we write
\[\gamma\preceq \alpha \text{ \ if 
$\lambda(\gamma)=\lambda(\alpha)$ and $w(\gamma)\leq w(\alpha)$ in Bruhat order.}\]
Here $\lambda(\gamma)$ is the partition obtained by resorting the
parts of $\gamma$. Also $w(\gamma)$ is the shortest length
permutation that sends $\lambda(\gamma)$ to $\gamma$. (Strong) Bruhat order refers to
the ordering on permutations obtained as the closure of the relation
$w\leq wt_{ij}$ if $\ell(wt_{ij})=\ell(w)+1$ and $t_{ij}$ is a transposition.

\begin{Theorem}
\label{cor:vertexkeys}
If $\beta\preceq \alpha$ then $\beta$ is a vertex of ${\sf Newton}(\kappa_{\alpha})$.
\end{Theorem}

\begin{Conjecture}
\label{conj:vertexkeys}
The converse of Theorem~\ref{cor:vertexkeys} holds.
\end{Conjecture}

Our proof of Theorem~\ref{cor:vertexkeys} uses
the other specialization of interest, namely $E_{\alpha}(X;q=0,t=0)$. 
Let
\[\widehat\pi_i:=\pi_i-{\rm id}\]
and define the {\bf Demazure atom} $A_{\alpha}=x^{\alpha}$ 
if $\alpha$ is weakly decreasing. Otherwise
$A_{\alpha}=\widehat{\pi_i}(A_{\widehat{\alpha}})$ where
$\widehat\alpha$ is defined as in (\ref{eqn:keypidef}). By the way,
\begin{Conjecture}
\label{conj:AisSNP}
$A_{\alpha}$ has SNP.
\end{Conjecture}
That $E_{\alpha}(X;q=0,t=0)=A_{\alpha}$ is \cite[Theorem~1.1]{Mason}.

The five conjectures above, namely Conjectures~\ref{conj:keytopeineq}, 
\ref{conj:kappaSNP}, \ref{conj:keyBruhatCondition}, 
\ref{conj:vertexkeys} and~\ref{conj:AisSNP} have been checked for
$|\alpha|\leq 7$ where $\alpha$ has at most three
parts of size zero.

We will also use
\begin{equation}
\label{eqn:kappa2atom}
\kappa_{\alpha}=\sum_{\gamma\preceq \alpha} A_{\gamma}
\end{equation}
One reference for (\ref{eqn:kappa2atom}) is \cite[Section~1]{Mason}; a proof is found in \cite[Lemma~3.5]{Pun}.

\begin{Proposition}
\label{prop:nesting}
Suppose $\beta\preceq\alpha$. Let $\lambda=\lambda(\beta)=\lambda(\alpha)$.
Then
\[\{\lambda\}\subseteq {\sf Newton}(\kappa_{\beta})\subseteq
{\sf Newton}(\kappa_{\alpha})\subseteq {\mathcal P}_{\lambda}\subseteq {\mathbb R}^n,\]
where $n$ is the position of the last nonzero part of $\alpha$.
\end{Proposition}
\begin{proof}
Using (\ref{eqn:kappa2atom}) twice, we have
\[
\kappa_\alpha = \sum_{\gamma \preceq \alpha} A_\gamma  = \sum_{\gamma \preceq \beta} A_\gamma + \sum_{\substack{\gamma \preceq \alpha \\ \gamma \not\preceq \beta}} A_\gamma = \kappa_\beta +  \sum_{\substack{\gamma \preceq \alpha \\ \gamma \not\preceq \beta}} A_\gamma.\]
Since each $A_{\gamma}$ is
monomial positive \cite[Theorem~1.1]{Mason},  
\[{\sf Newton}(\kappa_\beta)\subseteq {\sf Newton}(\kappa_{\alpha}).\]

Now, $\lambda$ is $\preceq$-minimum among rearrangements
of $\lambda$. By definition $\kappa_{\lambda}=x^{\lambda}$. This explains the leftmost containment. 

Let $\lambda^{\rm rev}=(0,0,\ldots,0,\ldots,\lambda_3,\lambda_2,\lambda_1)\in {\mathbb Z}^n$. Then $\lambda^{\rm rev}$ is the $\preceq$-maximum among 
rearrangements of $\lambda$ in ${\mathbb Z}^n$. Also, we have
$\kappa_{\lambda^{\rm rev}} = s_\lambda$ (see, e.g., \cite[Section~4]{Mason} and references therein). However we know 
${\sf Newton}(s_{\lambda})={\mathcal P}_{\lambda}$. 
\end{proof}

\begin{Lemma}
\label{lemma:PinQvertex}
Suppose ${\mathcal P}$ and ${\mathcal Q}$ are polytopes such that ${\mathcal P}\subseteq {\mathcal Q}$.
If $v$ is a vertex of ${\mathcal Q}$ and $v\in {\mathcal P}$, then $v$ is a vertex of ${\mathcal P}$.
\end{Lemma}
\begin{proof}
$v$ is a vertex of ${\mathcal Q}$ if and only if 
there is a separating hyperplane $H$, i.e., there exists a vector $c$ such that
${\bf c}'v<{\bf c}'y$ for all $y\in {\mathcal Q}$. Since 
${\mathcal P}\subseteq {\mathcal Q}$, $H$
works for ${\mathcal P}$ also.
\end{proof}

\noindent
\emph{Proof of Theorem~\ref{cor:vertexkeys}:}
 Now, 
\[\kappa_{\alpha}=x^{\alpha}+\text{(positive sum of monomials)};\]
see, e.g., \cite[Corollary~7]{Reiner.Shimozono}. 
Hence, $\alpha$ is in ${\sf Newton}(\kappa_{\alpha})$. 
By Proposition~\ref{prop:nesting}, 
\[\beta\in {\sf Newton}(\kappa_{\beta})\subseteq {\sf Newton}(\kappa_{\alpha}) \text{\ if $\beta\preceq\alpha$.}\]

Again applying Proposition~\ref{prop:nesting} we have that
${\sf Newton}(\kappa_{\alpha})\subseteq {\mathcal P}_{\lambda(\alpha)}$. Now we are done by combining
Proposition~\ref{prop:generalthing}(II) and Lemma~\ref{lemma:PinQvertex}.
\qed

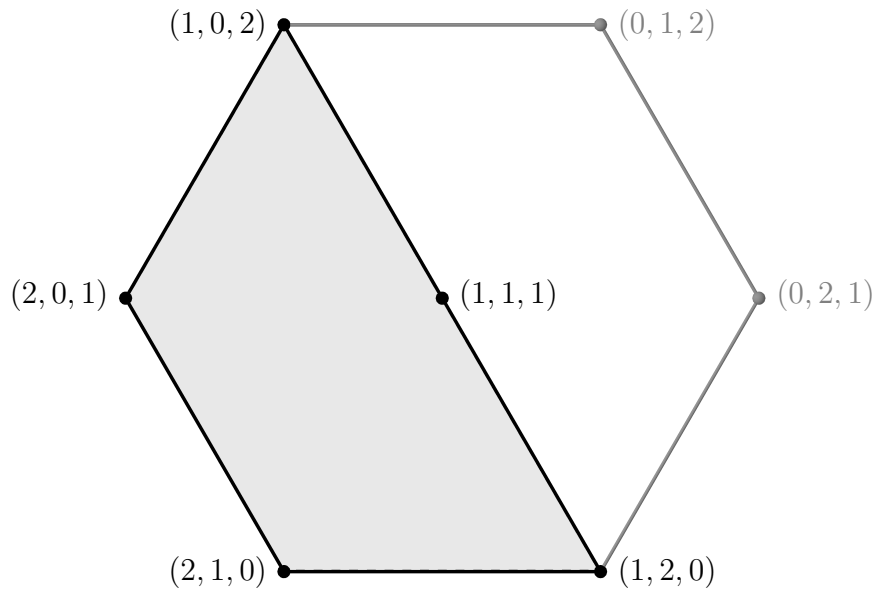
\begin{figure}
\begin{picture}(200,200)
\put(0,0){\includegraphics[width=3.5in]{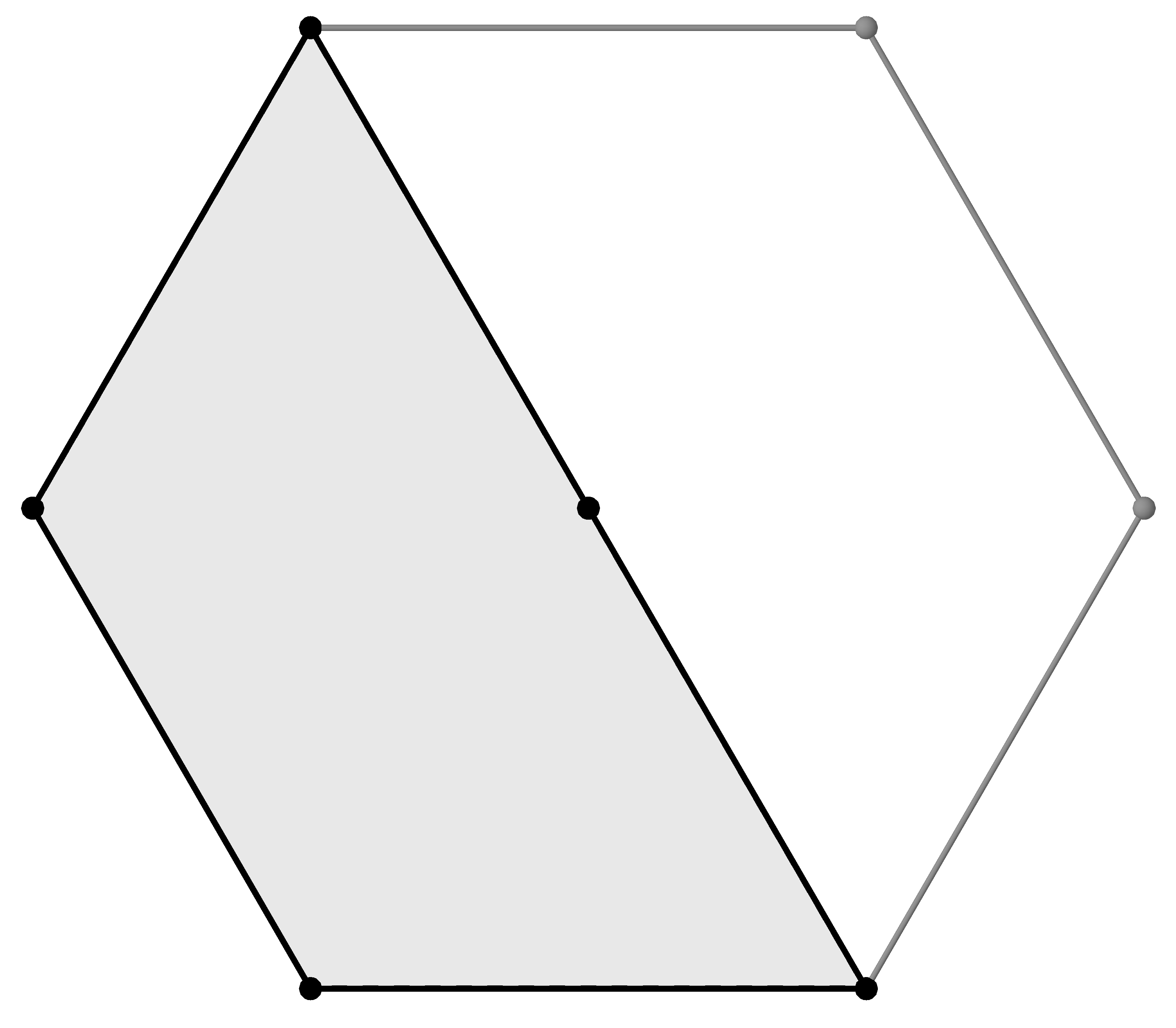}}
\put(23,212){$(1,0,2)$}
\put(-37,110){$(2,0,1)$}
\put(23,5){$(2,1,0)$}
\put(193,5){$(1,2,0)$}
\put(253,110){\color{gray}$(0,2,1)$}
\put(193,212){\color{gray}$(0,1,2)$}
\put(133,110){$(1,1,1)$}
\end{picture}
\caption{The permutahedron for $\lambda=(2,1,0)$. The shaded region is ${\sf Newton}(\kappa_{1,0,2})$. See Proposition~\ref{prop:nesting}.}
\label{fig:nestedkeys}
\end{figure}

\section{Quasisymmetric Functions}
\label{sec:quasi}
A power series $f\in {\mathbb Z}[[x_1,x_2,\ldots]]$ is
{\bf quasisymmetric} if
\[[x_1^{a_1}x_2^{a_2}\cdots x_k^{a_k}]f=
[x_{i_1}^{a_1}x_{i_2}^{a_2}\cdots x_{i_k}^{a_k}]f\] 
for any
natural numbers $i_1<i_2<\ldots<i_k$.
As with ${\sf Sym}$, define a quasisymmetric function $f$ to be SNP if $f(x_1,x_2,\ldots,x_m,0,0,\ldots)$ is SNP for all $m\geq 1$. 
In view of Remark~\ref{remark:quasistability}, $f$ is SNP if
$f(x_1,x_2,\ldots,x_m,0,0,\ldots)$ is SNP for any $m\geq \deg(f)$. 

Let $\alpha=(\alpha_1,\ldots,\alpha_k)\in {\mathbb Z}_{>0}^k$. The {\bf monomial quasisymmetric function} is defined as
\[M_{\alpha}=\sum_{i_1<i_2<\ldots<i_k} x_{i_1}^{a_1}\cdots
x_{i_k}^{a_k}.\]
Let ${\sf QSym}$ be the ${\mathbb Q}$-span of all $M_{\alpha}$.

\begin{Example}[$M_{\alpha}$ need not be SNP]
$M_{(2)}=p_2=x_1^2+x_2^2+\cdots$ does not have SNP. \qed
\end{Example}

Another basis of ${\sf QSym}$ is given by
{\bf Gessel's fundamental quasisymmetric functions}
\begin{equation}
\label{eqn:gesselfun}
F_{\alpha}=\sum_{\beta\to\alpha} M_{\beta}.
\end{equation}
Here, $\beta\to\alpha$ means that $\alpha$ is obtained
by successively adding adjacent parts of $\beta$. 

For a composition $\gamma$, let $\gamma^+$ be the composition formed by removing parts 
of size zero from $\gamma$.

\begin{Theorem} 
\label{prop:quasi}
${\sf Newton}(F_{\alpha}(x_1,\ldots,x_n)) = {\sf 
Newton}(M_{\alpha}(x_1,\ldots,x_n))\subset {\mathbb R}^n$. The vertices of this 
polytope are $\{ \gamma \in \mathbb{Z}_{\geq 0}^n : \gamma^+ = \alpha \}$.
\end{Theorem}
\begin{proof}
Each $M_{\beta}$ is a positive sum of monomials. Also, $\alpha\to\alpha$ so 
$M_{\alpha}$ appears in the expansion (\ref{eqn:gesselfun}). Therefore,
\[{\sf Newton}(F_{\alpha}(x_1,\ldots,x_n)) \supseteq {\sf 
Newton}(M_{\alpha}(x_1,\ldots,x_n)).\]

Suppose ${\beta}=(\beta_1,\beta_2,\ldots,
\beta_k)\in {\mathbb Z}_{>0}^k$ and $\widehat \beta\to\beta$ where 
\[{\widehat\beta}=(\beta_1,\beta_2,\ldots,\beta_i',\beta_{i}'',\ldots,\beta_k)\in {\mathbb Z}^{k+1}_{>0}\]
and $\beta_i=\beta_i'+\beta_{i}''$.

We wish to show
\begin{equation}
\label{eqn:mar4indstep}
{\sf Newton}(M_{{\widehat\beta}}(x_1,\ldots,x_n)) \subseteq {\sf 
Newton}(M_{\beta}(x_1,\ldots,x_n)).
\end{equation}
By induction, this implies the remaining containment
\[{\sf Newton}(F_{\alpha}(x_1,\ldots,x_n)) \subseteq {\sf 
Newton}(M_{\alpha}(x_1,\ldots,x_n)).\]

Suppose 
\[{\widetilde{\beta}}=({\widetilde{\beta_1}},\ldots,{\widetilde{\beta_n}})
\in {\mathbb Z}_{\geq 0}^n
\text{\ \ where $({\widetilde\beta})^+={\widehat\beta}$.}\]
Thus, 
\[{\widetilde{\beta}}=
(0,\ldots,0,\beta_1,0,\ldots,0,\beta_2,\ldots,\beta_i',0,\ldots,0,\beta_i'',
0,\ldots,0,\ldots,\beta_k,0,\ldots,0)\]
where we are depicting the additional $0$'s inserted between components of
${{\widehat\beta}}$ to obtain ${\widetilde{\beta}}$. In particular,
$x^{{\widetilde{\beta}}}$ appears in $M_{{\widehat\beta}}$.

Now let 
\[\beta^{\circ}=(0,\ldots,0,\beta_1,0,\ldots,0,\beta_2,\ldots,\beta_i,0,\ldots,0,0,0,
\ldots,0,\ldots,\beta_k,0,\ldots,0)\]
and
\[\beta^{\bullet}=(0,\ldots,0,\beta_1,0,\ldots,0,\beta_2,\ldots,0,0,\ldots,0,\beta_i,0,
\ldots,0,\ldots,\beta_k,0,\ldots,0).\]
That is $\beta^{\circ}$ and $\beta^{\bullet}$ differ from ${\widetilde{\beta}}$
only by replacing $\beta_i'$ and $\beta_i''$ by $\beta_i$, respectively.

Since $\beta_i',\beta_{i}''\geq 0$, we have that
\[{\widetilde{\beta}}=\frac{\beta_i'}{\beta_i}\beta^{\circ}+
\frac{\beta_i^{''}}{\beta_i}\beta^{\bullet}\]
is a convex combination. This proves
(\ref{eqn:mar4indstep}) and hence the asserted equality of Newton polytopes.

Every monomial of $M_{\alpha}(x_1,\ldots,x_n)$ is a monomial of 
$m_{\alpha}(x_1,\ldots,x_n)$. Therefore,
\[{\sf Newton}(M_{\alpha}(x_1,\ldots,x_n))\subseteq 
{\sf Newton}(m_{\alpha}(x_1,\ldots,x_n)).\]
Recall,
\[{\sf Newton}(m_{\alpha}(x_1,\ldots,x_n))={\mathcal P}_{\lambda(\alpha)}
\subseteq {\mathbb R}^n.\] 
One knows
the vertices of ${\mathcal P}_{\lambda(\alpha)}$ are all rearrangements of
$\alpha$ (thought of as a vector in ${\mathbb Z}^n_{\geq 0}$, where we concatenate
$0$'s as necessary); cf.~Proposition~\ref{prop:generalthing}(II). Thus,
every exponent vector of $m_{\alpha}(x_1,\ldots,x_n)$ is also a vertex of
${\mathcal P}_{\lambda(\alpha)}$. Hence to obtain the final claim of the theorem 
we may appeal to 
Lemma~\ref{lemma:PinQvertex}.
\end{proof}

\begin{Example}[$F_{\alpha}$ need not be SNP]
One can also
check that
\[F_{(2,2)} = M_{(2,2)} + M_{(2,1,1)} + M_{(1,1,2)} + M_{(1,1,1,1)}.\] 
Thus, 
$(0,1,2,1) = \frac{1}{2}(0,2,2,0) + \frac{1}{2}(0,0,2,2) 
\in {\sf Newton}(F_{(2,2)})$.
However, $(0,1,2,1)$ is not an exponent vector of $F_{(2,2)}$. 
Hence $F_{(2,2)}$ is not SNP. \qed
\end{Example}

J.~Hagland-K.~Luoto-S.~Mason-S.~van Willigenburg \cite{quasischur}
introduced the {\bf quasisymmetric Schur polynomial}:
\[S_{\alpha}=\sum_{\gamma} A_{\gamma}\]
where the sum is over all compositions $\gamma$ such that $\gamma^+=\alpha$ and where $\gamma^+$ is the composition
$\gamma$ with any $0$ parts removed. ${\sf QSym}$ is also spanned
by $\{S_{\alpha}\}$. Also, recall $A_{\gamma}$ is the Demazure atom defined in
Section~\ref{subsection:keysandatoms}. 

Many aspects of quasi-Schur theory 
are parallel to Schur theory \cite{quasischur}. For instance,
consider the transition between the $S$ and $M$ bases
of ${\sf QSym}$:
\[S_{\alpha}=\sum_{\beta}{\overline {\sf K}}_{\alpha,\beta}M_{\beta}.\]
It is proved in \emph{loc.~cit.} that ${\overline {\sf K}}_{\alpha,\beta}$ counts \emph{composition tableaux}. Hence
${\overline {\sf K}}_{\alpha,\beta}$ is an analogue of the Kostka coefficient.
However, there are divergences from the perspective of Newton polytopes as seen in the next three examples:

\begin{Example}[$S_{\alpha}$ need not be SNP] 
An example is $S_{(2,1,3)}$. In at least four variables, $x_1x_2^2x_3^2x_4$ does not appear but $x_1^2x_2^2x_3^2$ and $x_2^2x_3^2x_4^2$ both do.  Nonetheless, it should be interesting to describe the Newton polytope, and to characterize when $S_{\alpha}$ is SNP.\qed
\end{Example}

\begin{Example}
In the symmetric function case, 
\[{\sf Newton}(s_{\lambda}(x_1,\ldots,x_n))={\sf Newton}(m_{\lambda}(x_1,\ldots,x_n))={\mathcal P}_{\lambda}\subset {\mathbb R}^n.\] 
However,
\[(0,0,2,2)\in {\sf Newton}(S_{(1,3)}(x_1,x_2,x_3,x_4)) \text{\ \  but \ \  
 $(0,0,2,2)\not\in {\sf Newton}(M_{(1,3)}(x_1,x_2,x_3,x_4))$.}\]
Hence ${\sf Newton}(S_{\alpha}(x_1,\ldots,x_n)) \neq {\sf Newton}(M_{\alpha}(x_1,\ldots,x_n))$ in general.  \qed
\end{Example}

\begin{Example}
We may define a dominance order $\preceq_D'$ on strict compositions by 
$\alpha \preceq_D' \beta$ if ${\sf Newton}(M_{\alpha})\subseteq {\sf Newton}(M_{\beta})$. The above example shows that 
${\overline {\sf K}}_{\alpha,\beta}>0$ if and only if
$\beta\preceq_D' \alpha$ is not generally true. This is in contrast with (\ref{eqn:containDominance}).\qed
\end{Example}

\section{Schubert polynomials and variations}
\label{sec:schub}

\subsection{The Schubert SNP conjectures}
\label{sec:schubsnpsec}
A.~Lascoux--M.-P.~Sch\"{u}tzenberger \cite{LS1} 
introduced {\bf Schubert polynomials}. If $w_0=n \ n-1 \ \cdots 2 \ 1 $ (in one-line notation)
is the longest length permutation 
in $S_n$ then 
\[{\mathfrak S}_{w_0}(x_1,\ldots,x_n):=x_1^{n-1}x_2^{n-2}\cdots x_{n-1}.\] 
Otherwise, $w\neq w_0$ and 
there exists $i$ such that $w(i)<w(i+1)$. Then one sets
\[{\mathfrak S_w}(x_1,\ldots,x_n)=\partial_i {\mathfrak S}_{ws_i}(x_1,\ldots,x_n), \text{\ where
$\partial_i f:= \frac{f-s_if}{x_i-x_{i+1}}$,}\]
and $s_i$ is the simple transposition swapping $i$ and $i+1$. 
Since $\partial_i$ satisfies
\[\partial_i\partial_j=\partial_j\partial_i \text{\ for $|i-j|>1$, and \ } 
\partial_i\partial_{i+1}\partial_i=\partial_{i+1}\partial_i\partial_{i+1},\]
the above description of ${\mathfrak S}_w$ is well-defined. In addition,
under the inclusion 
$\iota: S_n\hookrightarrow S_{n+1}$ defined by
$w(1)\cdots w(n) \mapsto w(1) \ \cdots w(n) \ n+1$,
we have ${\mathfrak S}_w={\mathfrak S}_{\iota(w)}$. Thus one unambiguously
refers to ${\mathfrak S}_w$ for each $w\in S_{\infty}=\bigcup_{n\geq 1} S_n$. 

\begin{Conjecture}
\label{conj:main2} 
${\mathfrak S}_w$ has SNP.
\end{Conjecture}

We have checked Conjecture~\ref{conj:main2} for all $w\in S_n$ where $n\leq 8$.

Let $X=\{x_1,x_2,\ldots\}$ and $Y=\{y_1,y_2,\ldots\}$.
The {\bf double Schubert polynomial} ${\mathfrak S}_{w}(X;Y)$ is defined
by setting
\[{\mathfrak S}_{w_0}(X;Y)=\prod_{i+j\leq n} (x_i-y_j)\]
and recursively determining ${\mathfrak S}_w(X;Y)$ for $w\neq w_0$ precisely as 
for ${\mathfrak S}_w(X)$. 

We have also checked for $n\leq 5$ (and many other cases) that:

\begin{Conjecture}
\label{conj:double}
${\mathfrak S}_{w}(X;Y)$ is SNP.
\end{Conjecture}
 
Since ${\mathfrak S}_{w}(X;0)={\mathfrak S}_{w}(X)$, Conjecture~\ref{conj:double}
implies Conjecture~\ref{conj:main2}.

\begin{Example}[$\partial_i$ and $\pi_i$ does not preserve SNP]\label{exa:partialsubtle}
This polynomial is SNP:
\[f=x_1^4+x_1^3 x_2 +x_1^2 x_2^2 +2x_1 x_2^3.\]
However
\[\partial_1(f)=x_1^3+x_2^3\]
is not SNP. 

Since $\pi_i(g)=\partial_i(x_i\cdot g)$, if we
set 
\[g=x_1^3+x_1^2x_2+x_1x_2^2+2x_2^3\] 
we have
$\pi_1(g)=\partial_1(f)$. Hence, $\pi_i$
does not preserve SNP.
\qed  
\end{Example}
\begin{Example}[Double Schubert polynomials are generalized resultants]\label{exa:doubleschubres}
Pick $w$ to be the ``dominant'' permutation 
$n+1 \ n+2 \ \cdots n+m \ 1 \ 2 \ \cdots \ n \in S_{n+m}$. Then
\[{\mathfrak S}_w(X;Y)=\prod_{i=1}^n\prod_{j=1}^m (x_i-y_j).\]
(One reference is \cite[Proposition~2.6.7]{Manivel}.)
This has the same Newton polytope as $R(f,g)$. Thus Conjecture~\ref{conj:double} is
proposes a generalization of Theorem~\ref{prop:resultantSNP}.\qed
\end{Example}

A.~Lascoux--M.-P.~Sch\"utzenberger also introduced the family of
 {\bf Grothendieck polynomials} \cite{lascoux.schutzenberger}.
These polynomials are defined using
\[\overline{\pi}_i(f)=\partial_i((1-x_{i+1})f).\]
For $w_0\in S_n$ declare
\[{\mathfrak G}_{w_0}(X)=x_1^{n-1}x_2^{n-2}\cdots x_{n-1}.\] 
If $w\in S_n$ and $w\neq w_0$, let
\[{\mathfrak G}_w(X)={\overline{\pi}}_i({\mathfrak G}_{ws_i})\] 
if $i$ is an ascent of $w$.
This is an inhomogenous analogue of the Schubert polynomial since
\[{\mathfrak G}_w(X)={\mathfrak S}_w(X)+\mbox{(higher degree terms)}.\]
Like the Schubert polynomials, ${\mathfrak G}_w={\mathfrak G}_{\iota(w)}$, where $\iota:S_n\hookrightarrow
S_{n+1}$ is the natural inclusion. Hence it make sense to
define ${\mathfrak G}_w$ for $w\in S_{\infty}$.
 
\begin{Conjecture}
\label{conj:grothendieckSNP}
${\mathfrak G}_w$ has SNP.
\end{Conjecture}

Conjecture~\ref{conj:grothendieckSNP} has been exhaustively checked
for $n\leq 7$. Conjecture~\ref{conj:grothendieckSNP} generalizes Conjecture~\ref{conj:main2} since  
\[{\sf Newton}({\mathfrak S}_w)={\sf Newton}({\mathfrak G}_w)
\cap \left\{(\alpha_1,\ldots,\alpha_n)\in {\mathbb R^n}: \sum_{i=1}^n \alpha_i=\#{\sf D}_w\right\}.\]

Grothendieck polynomials arise in combinatorial $K$-theory. Another family
of polynomials from this topic was introduced by
A.~Lascoux in \cite{Lascoux:trans}. He defines $\Omega_{\alpha}$
for $\alpha=(\alpha_1,\alpha_2,\ldots) \in {\mathbb Z}_{\geq 0}^{\infty}$ by replacing $\pi_i$ in the definition
of $\kappa_{\alpha}$ with 
\[\tau_i(f)=\partial_i(x_i(1-x_{i+1})f).\]
The initial condition is
$\Omega_{\alpha}=x^{\alpha}(=\kappa_{\alpha}), \text{\ \ if $\alpha$ is weakly decreasing.}$
$\Omega_{\alpha}$ is an inhomogeneous analogue of $\kappa_{\alpha}$. 

\begin{Conjecture}
\label{conj:mar7bhh}
${\Omega}_\alpha$ has SNP.
\end{Conjecture}

The {\bf Lascoux atom} $\mathcal{L}_\alpha$ is defined \cite{Monical} by replacing $\pi_i$ in the definition of $\kappa_\alpha$ with \[ \widehat{\tau}_i(f) = (\tau_i-1)f. \]

\begin{Conjecture}
\label{conj:mar7yyy}
$\mathcal{L}_\alpha$ has SNP.
\end{Conjecture}

${\mathcal L}_{\alpha}$ is an inhomogeneous analogue of $A_{\alpha}$.

Conjectures~\ref{conj:mar7bhh} and~\ref{conj:mar7yyy}
have been verified for $|\alpha|\leq 7$ where
$\alpha$ has at most three parts of size zero.

\subsection{Stanley polynomials and the stable limit of Conjecture~\ref{conj:main2}}

For $w\in S_{n}$, let $1^t\times w\in S_{t+n}$
be the permutation defined by $1^t\times w(i)=i$ for $1\leq i\leq t$ and
$1^t\times w(i)=n+i$ for $t+1\leq i\leq t+n$. The {\bf Stanley symmetric polynomial} (also known as the 
{\bf stable Schubert polynomial}) is defined by
\[F_w=\lim_{t\to \infty}{\mathfrak S}_{1^t\times w}\in {\sf Sym}.\]
This power series is well-defined. 

$F_w$ was originally introduced by R.~P.~Stanley in \cite{Stanley1984}. 
Every $w\in S_n$
can be expressed as a product of simple transpositions 
\[w=s_{i_1}s_{i_2}\cdots s_{i_{\ell}}.\] 
When $\ell=\ell(w)$ is the number
of inversions of $w$, this factorization is {\bf reduced}. Then  $s_{i_1}\cdots s_{i_{\ell}}$, or equivalently
$(i_1,\ldots,i_{\ell})$, is a {\bf reduced word} for $w$. Let $\#{\rm Red}(w)$ be the number of reduced words of $w$.
In \emph{loc.~cit.} it is shown that 
\[\#{\rm Red}(w)=[x_1\cdots x_{\ell(w)}]F_w.\]

The next result is a ``stable limit'' version of Conjecture~\ref{conj:main1}.

\begin{Theorem}
\label{thm:Stanley}
$F_{w}\in {\sf Sym}$ is SNP.
\end{Theorem}

Our proof rests on:

\begin{Theorem}[Theorems 3.2, 4.1, \cite{Stanley1984}] For 
\[F_w = \sum_\lambda {\sf a}_{w,\lambda} s_\lambda,\] 
${\sf a}_{w,\lambda} \geq 0$ and there exists $\lambda(w)$ and $\mu(w)$ such that if ${\sf a}_{w,\lambda} \neq 0$, then $\lambda(w) \leq_D \lambda \leq_D \mu(w)$.\label{thm:stan84}
\end{Theorem}

\noindent
\emph{Proof of Theorem~\ref{thm:Stanley}:} Combine Theorem~\ref{thm:stan84} and Proposition~\ref{prop:generalthing}(III).\qed

\begin{Corollary}
\label{cor:skewSNP}
Any skew-Schur polynomial $s_{\lambda/\mu}(X)$ has SNP.
\end{Corollary}
\begin{proof}
To every skew shape $\lambda/\mu$ there is a $321$-avoiding permutation $w_{\lambda/\mu}$ with the property that
$F_{w_{\lambda/\mu}}(X)=s_{\lambda/\mu}$ \cite{BJS}. 
Now apply Theorem~\ref{thm:Stanley}.
\end{proof}

Let 
\[S_{\infty,\ell}=\{w\in S_{\infty}:\ell(w)=\ell\}.\]
Declare 
\[u\preceq_{D} v \text{\ for $u,v\in S_{\infty,\ell}$
\ if ${\sf Newton}({\mathfrak S}_u)\subseteq {\sf Newton}({\mathfrak S}_v)$.}\]

Given a partition $\lambda=(\lambda_1\geq \lambda_2\geq \ldots\geq \lambda_k>0)$, define $w_{\lambda,k}\in S_{\lambda_1+k}$ 
to be the unique permutation that satisfies \[w_{\lambda,k}(i)=\lambda_{k-i+1}+i \mbox{\ for $1\leq i\leq k$}\] 
and is {\bf Grassmannian}, i.e., it has at most one descent, at position $k$. Then one has 
\[{\mathfrak S}_{w_{\lambda,k}}=s_{\lambda}(x_1,\ldots,x_k).\]
We now show that $(S_{\infty},\preceq_D)$
 extends $({\sf Par}(n),\leq_D)$:

\begin{Proposition}
Suppose $\lambda,\mu\in {\sf Par}(n)$ and let $k=\max\{\ell(\lambda),\ell(\mu)\}$. 
Then $\lambda\leq_D\mu$ if and only if
$w_{\lambda,k}\preceq_D w_{\mu,k}$.
\end{Proposition}
\begin{proof}

Since ${\mathfrak S}_{w_{\lambda,k}}(x_1,\ldots,x_k)=s_{\lambda}(x_1,\ldots,x_k)$
and ${\mathfrak S}_{w_{\mu,k}}(x_1,\ldots,x_k)=s_{\mu}(x_1,\ldots,x_k)$, 
\[{\mathcal P}_{\lambda}={\sf Newton}(s_{\lambda}(x_1,\ldots,x_k))
={\sf Newton}({\mathfrak S}_{w,\lambda}(x_1,\ldots,x_k)) \subseteq \mathbb{R}^k.\]
The same statement holds where we replace $\lambda$ by $\mu$.
Now apply Rado's theorem (\ref{eqn:containDominance}).
\end{proof}

Figure~\ref{fig:poset} shows part of $(S_{\infty,2},\preceq_D)$. From this one can see 
that the poset is not graded, just like dominance order $\leq_D$ on partitions.  
Unlike $\leq_D$, it is not a lattice: in Figure~\ref{fig:poset}, the elements $231456$ 
and $312456$ do not have a unique least upper bound as $142356$ and $214356$ are 
incomparable minimal upper bounds.

\begin{Theorem}
Every two elements $u,v\in S_{\infty,\ell}$ have an upper bound under $\preceq_D$.
\end{Theorem}
\begin{proof}
Suppose $\{\alpha_i\}$ and $\{\beta_j\}$ are the exponent vectors of ${\mathfrak S}_u$ and
${\mathfrak S}_v$, respectively. It suffices to show
there exists $w\in S_{\infty,\ell}$ such that
\[{\mathfrak S}_w=\sum_i x^{\alpha_i}+\sum_{j} x^{\beta_j}+\text{(positive sum of monomials)}.\]

We first show that there is a $F_w$ such that each
$s_{\lambda(\alpha_i)}$ and $s_{\lambda(\beta_j)}$ appear
(possibly with multiplicity). A theorem of S.~Fomin-C.~Greene \cite{Fomin.Greene} states that
\[F_w=\sum_{\nu} {\sf a}_{w,\nu}s_{\nu}\]
where ${\sf a}_{w,\nu}$ is the number of semistandard tableaux of shape $\nu$ such that the top-down, right-to-left
reading word is a reduced word for $w$. Let 
\[w=s_1 s_{3} s_{5}\cdots s_{2\ell-1}.\]
Clearly this is a reduced word. All reduced words of $w$
are obtained by permuting the simple transpositions.

Filling \emph{any} shape of size $\ell$ by successively placing 
$1,3,5,\ldots,2\ell-1$ along rows in left to right order gives a semistandard tableaux. Thus every $s_{\mu}$ where $\mu\vdash \ell$
appears in $F_w$. In particular each $s_{\lambda(\alpha_i)}$ and each $s_{\lambda(\beta_j)}$ appears. Since 
$x^{\lambda(\alpha_i)}$ appears in $s_{\lambda(\alpha_i)}$, by symmetry of $s_{\lambda(\alpha_i)}$, $x^{\alpha_i}$ appears
as well. That is, $x^{\alpha_i}$ appears in $F_w$. Similarly $x^{\beta_j}$ appears in $F_w$.

By definition, for any monomial $x^\gamma$ appearing in $F_w$, there is a finite $N_\gamma$ such that $x^\gamma$ appears in ${\mathfrak S}_{1^{N_\gamma}\times w}$. It suffices to pick $N$ larger than all $N_{\alpha_i}$ and $N_{\beta_j}$.
\end{proof}

\begin{figure}
\begin{picture}(300,350)
\put(0,0){\includegraphics[height=4.5in]{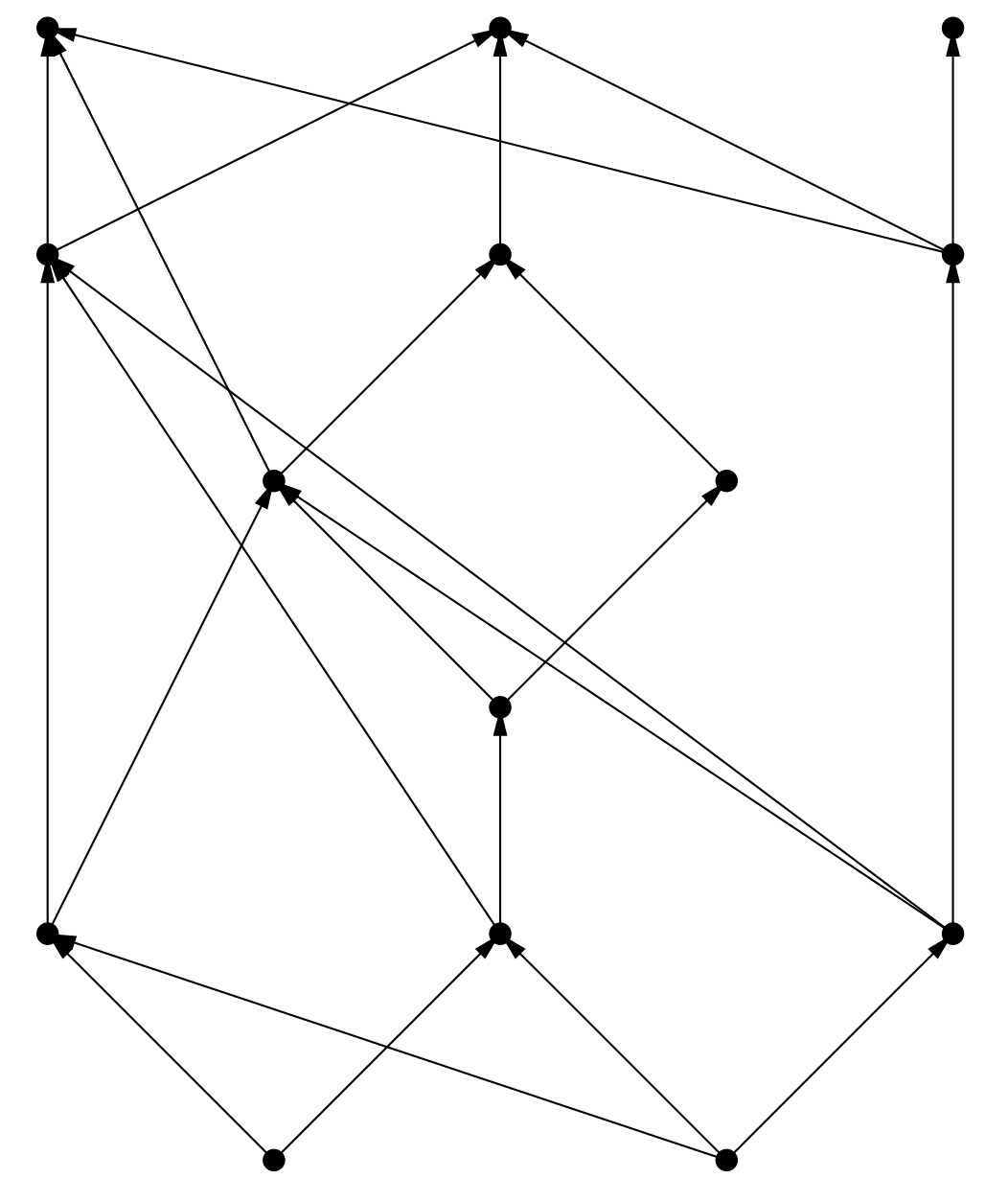}}
\put(33,7){$312456$}
\put(202,7){$231456$}
\put(-28,67){$142356$}
\put(93,67){$214356$}
\put(262,67){$134256$}
\put(93,130){$213546$}
\put(30,190){$132546$}
\put(202,190){$213465$}
\put(-28,252){$125346$}
\put(92,252){$132465$}
\put(262,252){$124536$}
\put(-28,314){$123645$}
\put(92,314){$124365$}
\put(218,314){$123564$}
\end{picture}

\caption{The $S_6$ part of the Hasse diagram of $(S_{\infty,2},\preceq_D)$}  
\label{fig:poset}
\end{figure}
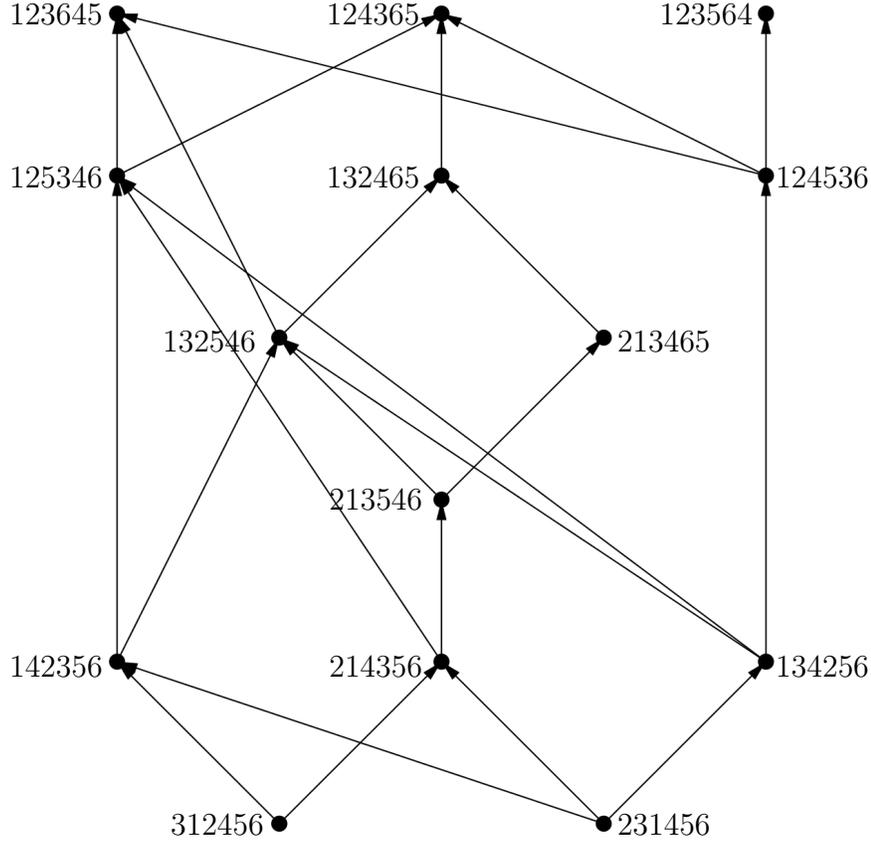

\subsection{Conjectural inequalities for the ${\sf Newton}({\mathfrak S}_w)$}
Let 
\[{\sf D}_w=\{(i, j) : 1 \leq i, j \leq n, w(i) > j \text{\ and \ } w^{-1}(j) > i\}\] 
be the {\bf Rothe diagram} of a permutation $w\in S_n$. 

\begin{Conjecture}
\label{conj:main1}
${\mathcal S}_{{\sf D}_{w}}={\sf Newton}({\mathfrak S}_w)$.
\end{Conjecture}

This has been checked for all $w \in S_8$, as well as many larger instances.
Notice that Conjecture~\ref{conj:main1} is equivalent to the assertion that
$w\preceq_D v$ if and only if $\theta_{{\sf D}_w}(S)\leq \theta_{{\sf D}_v}(S)$
for all $S\subseteq [n]$. 

\begin{Example}
\label{example:21543}
Suppose $w=21543$, the Rothe diagram ${\sf D}_w$ is given by
\begin{center}
\begin{tikzpicture}[scale=0.5]
\draw (0,0) rectangle (5,5);

\draw[fill=gray!30] (0,4) rectangle (1,5);
\draw[fill=gray!30] (2,2) rectangle (3,3);
\draw[fill=gray!30] (3,2) rectangle (4,3);
\draw[fill=gray!30] (2,1) rectangle (3,2);

\filldraw (1.5,4.5) circle (.5ex);
\draw[line width = .2ex] (1.5,0) -- (1.5,4.5) -- (5,4.5);
\filldraw (0.5,3.5) circle (.5ex);
\draw[line width = .2ex] (0.5,0) -- (0.5,3.5) -- (5,3.5);
\filldraw (4.5,2.5) circle (.5ex);
\draw[line width = .2ex] (4.5,0) -- (4.5,2.5) -- (5,2.5);
\filldraw (3.5,1.5) circle (.5ex);
\draw[line width = .2ex] (3.5,0) -- (3.5,1.5) -- (5,1.5);
\filldraw (2.5,0.5) circle (.5ex);
\draw[line width = .2ex] (2.5,0) -- (2.5,0.5) -- (5,0.5);
\end{tikzpicture}
\end{center}
One can check that the defining inequalities are 
\[\alpha_1+\alpha_2+\alpha_3+\alpha_4=4\]
\[\alpha_1\leq 3,\ \alpha_2\leq2,\ \alpha_3\leq 2,\ \alpha_4\leq 1\]
\[\alpha_1+\alpha_2\leq 4,\  \alpha_1+\alpha_3\leq 4,\  \alpha_1+\alpha_4\leq 4,\  \alpha_2+\alpha_3\leq 3,\  \alpha_{2}+\alpha_4\leq 3,
\alpha_3+\alpha_4\leq 3\]
\[\alpha_1+\alpha_2+\alpha_3\leq 4, \  \alpha_1+\alpha_2+\alpha_4\leq 4, \ \alpha_2+\alpha_3+\alpha_4\leq 3\]
\[\alpha_1+\alpha_2+\alpha_3+\alpha_4\leq 4.\]
together with $\alpha_i \geq 0$ for each $i$. The polytope is depicted in Section~1.\qed
\end{Example}

One can uniquely reconstruct $u\in S_{\infty}$ with the defining inequalities.
\begin{Proposition}
If $u,v\in S_{n}$ are of the same length and 
$\theta_{{\sf D}_{u}}(S)=\theta_{{\sf D}_v}(S)$ 
for all $S=\{i,i+1,\ldots,n\}$
where $1\leq i\leq n$, then $u=v$.
\end{Proposition}
\begin{proof}
Let 
\[c_i(\pi)=\#\{j:(i,j)\in {\sf D}_{\pi}\}.\] 
Thus $(c_1(\pi),c_2(\pi),\ldots)$ is the {\bf Lehmer code}
of $\pi$. The Lehmer code uniquely 
determines $\pi\in S_{\infty}$; see, e.g., \cite[Proposition~2.1.2]{Manivel}. Hence it suffices to show the
codes of $u$ and $v$ are the same. This follows from:
\[\sum_{j=1}^i c_j(u)=\ell-\theta_{{\sf D}_u}(\{i+1,i+2,\ldots,n\})=\ell-\theta_{{\sf D}_v}(\{i+1,i+2,\ldots,n\})=\sum_{j=1}^i c_j(v),\]
for $i=1,2,\ldots n-1$.
\end{proof}

The inequalities of ${\mathcal S}_{\sf D}$ are in general redundant. If 
\begin{equation}
\label{eqn:redux1}
\theta_{\sf D}(S)=\theta_{\sf D}(T) \text{\ and $S\supseteq T$}
\end{equation}
then the inequality 
\[\sum_{i\in T}\alpha_i\leq \theta_{\sf D}(T)\] 
is unnecessary. Similarly, if 
\begin{equation}
\label{eqn:redux2}
S= \sqcup_i T_i \text{\ \ and \ $\theta_{\sf D}(S)=\sum_i \theta_{\sf D}(T_i)$}
\end{equation} 
then the $S$-inequality is implied by the $T_i$ inequalities. 

\begin{Problem}
Give the minimal set of inequalities associated to ${\sf D}_w$ (or more generally, any ${\sf D}$).
\end{Problem}

\begin{Example}
Continuing Example~\ref{example:21543}, minimal inequalities are
\[\alpha_1+\alpha_2+\alpha_3+\alpha_4=4\]
\[\alpha_1\leq 3,\ \alpha_2\leq2,\ \alpha_3\leq 2,\ \alpha_4\leq 1\]
\[\alpha_1+\alpha_2\leq 4,\  \alpha_1+\alpha_3\leq 4,\  \alpha_2+\alpha_3\leq 3,\]
\[\alpha_2+\alpha_3+\alpha_4\leq 3,\]
combined with positivity. This minimization is obtained using reductions (\ref{eqn:redux1}) and (\ref{eqn:redux2}).
\end{Example}

\begin{Example}
If $w=23154$ then using the reductions (\ref{eqn:redux1}) and (\ref{eqn:redux2}) leaves:
\[\alpha_1+\alpha_2+\alpha_3+\alpha_4=3, \alpha_3+\alpha_4\leq 1,
\alpha_1+\alpha_3+\alpha_4\leq 2, \alpha_2+\alpha_3+\alpha_4\leq 2.\]
However, $\alpha_3+\alpha_4\leq 1$ is actually not necessary.\qed
\end{Example}

Given a polytope $P$, recall its \emph{Ehrhart polynomial}, denoted $L_P(t)$, 
is the polynomial such that for $t \in \mathbb{Z}_{\geq 1}$, $L_P(t)$ 
equals the number of lattice points in the polytope $tP$.  Ehrhart \cite{Ehrhart} showed that for a polytope of dimension $d$ in $\mathbb{R}^n$, $L_P(t)$ is in fact a polynomial of degree $d$. For more see, e.g., \cite{BR07}. 

\begin{Conjecture}
\label{conj:ehrhart}
If $L_{\mathcal{N}(\mathcal{S}_{\sf D})}(t) = c_dt^d + \cdots + c_0$, then $c_i > 0$ for $i = 0,\ldots,d.$ 
\end{Conjecture}

Conjecture~\ref{conj:ehrhart} also seems true for ${\mathcal S}_{\sf D}$ where ${\sf D}$ is arbitrary. We have exhaustively checked
this for $n=4$ and many random
cases for $n=5$.

Below we give some data about the positive dimensional Schubitopes ${\mathcal S}_{{\sf D}_w}$ for $w\in S_4$:
\begin{center}
\begin{tabular}[t]{c||c|c|c|c}
$w$ & $\mathfrak{S}_w$ & $\dim \mathcal{S}_{{\sf D}_w}$ & vertices of $\mathcal{S}_{{\sf D}_w}$ & $L_{\mathcal{S}_{{\sf D}_w}}(t)$ \\
\hline \hline
$1243$ & $x_1 + x_2 + x_3$ & $2$ & $(1,0,0), (0,1,0), (0,0,1)$ & $ \frac{1}{2} t^2 + \frac{3}{2}t + 1$ \\ \hline
$1324$ & $x_1 + x_2$ & $1$ & $(1,0), (0,1)$ & $t+1$ \\ \hline
$1342$ & $x_1x_2 + x_1x_3 + x_2x_3$ & $2$ & $(1,1,0),(1,0,1),(0,1,1)$ & $ \frac{1}{2} t^2 + \frac{3}{2}t + 1$ \\ \hline
$1423$ &$x_1^2 + x_1x_2 + x_2^2$ &$1$ & $(2,0),(0,2)$ & $2t+1$ \\ \hline
\multirow{ 2}{*}{$1432$}& $x_1^2x_2 + x_1x_2^2 + x_1^2x_3$ &\multirow{ 2}{*}{$2$} & $(2,0,1),(1,2,0),$ & \multirow{ 2}{*}{$ \frac{3}{2} t^2 + \frac{5}{2}t + 1$}\\ 
& $+x_1x_2x_3 + x_2^2x_3$ & & $(2,1,0),(0,2,1)$ &  \\ \hline
$2143$ & $x_1^2 + x_1x_2 + x_1x_3$ & $2$ & $(2,0,0), (1,1,0), (1,0,1)$ & $ \frac{1}{2} t^2 + \frac{3}{2}t + 1$ \\ \hline
$2413$ & $x_1^2x_2 + x_1x_2^2$ & $1$ & $(2,1), (1,2)$ & $t+1$ \\ \hline
$2431$ & $x_1^2x_2x_3 + x_1x_2^2x_3$ & $1$ & $(2,1,1), (1,2,1)$ & $t+1$ \\ \hline
$3142$ & $x_1^2x_2 + x_1^2x_3$ & $1$ & $(2,1,0), (2,0,1)$ & $t+1$ \\ \hline
$4132$ & $x_1^3x_2 + x_1^3x_3$ & $1$ & $(3,1,0), (3,0,1)$ & $t+1$ \\ 
\end{tabular}
\end{center}

\subsection{Relationship of the Schubitope to Kohnert's rule}

A.~Kohnert \cite{Kohnert} conjectured a combinatorial rule for ${\mathfrak S}_w(X)$.
 Starting from ${\sf D}_{w}$
one moves the rightmost box in any row up to the southmost unoccupied square of $n\times n$. Repeating this gives a finite
set of diagrams 
${\sf Koh}(w)=\{D\}$. Define 
\[{\tt Kohwt}(D)=\prod x_i^{\text{$\#$boxes of $D$ in row $i$}}.\] 
Let 
\[{\mathfrak K}_w=\sum_{D\in {\sf Koh}(w)} {\tt Kohwt}(D).\] 
The conjecture is that ${\mathfrak S}_w={\mathfrak K}_w$.
For a proof see work of R.~Winkel \cite{Winkel} and of S.~Assaf \cite{Assaf}. 
With this in hand, one obtains part of 
Conjecture~\ref{conj:main2}, i.e.,

\begin{Proposition}
\label{thm:onecontainment}
${\mathcal S}_{{\sf D}_w}\supseteq {\sf Newton}({\mathfrak K}_w)$.
\end{Proposition}
\begin{proof}
Consider a diagram $D\in {\sf Koh}(w)$ such that 
${\tt Kohwt}(D)=\alpha$. Each Kohnert move preserves the
number of boxes. Hence 
$\sum_{i=1}^n \alpha_i=\#{\sf D}_{w}$ holds.

Now fix a column $c$ and $S\subseteq [n]$. Compare the positions
of the boxes of $D$ to the boxes of ${\sf D}_{w}$. Let
\[T_{D,S,c}=\#\text{boxes of $D$ in the rows of $S$ and column $c$}.\]
Also, let $U_{D,S,c}$ be the number of pairs $(r,r')$,
\emph{with no coordinate repeated}, such that
\[r\in S, r'\not\in S, r<r', (r,c)\not\in {\sf D}_{w} \text{\ but $(r',c)\in {\sf D}_{w}$.}\]
Since Kohnert moves only bring boxes in from lower rows
into higher rows (i.e., boxes migrate from the south),
\[T_{D,S,c}\leq T_{{\sf D}_{w},S,c}+U_{{\sf D}_w,S,c}.\]
Now it is easy to check that 
\[\theta_{D}^c(S)=T_{{\sf D}_{w},S,c}+U_{{\sf D}_w,S,c}.\]
Since $\alpha_i$ counts the number of boxes in row $i$
of $D$, we have
\[
\sum_{i\in S}\alpha_i  =\sum_c T_{D,S,c}
 \leq \sum_c T_{{\sf D}_{w},S,c}+U_{{\sf D}_w,S,c}
 =\sum_c \theta_{D}(S)  =\theta_D(S),\]
as required.
\end{proof}

\begin{Remark}\label{remark:diffwithkohnert}\emph{
Unlike the computation of each $\theta_D^c(s)$,
the Kohnert moves are not column independent. Perhaps
surprisingly, Conjecture~\ref{conj:main2} says that the \emph{a priori} coarse upper bound on $\sum_{i\in s} \alpha_i$ captures all monomials appearing in the Schubert polynomial.}\qed
\end{Remark}

\begin{Remark}
\label{remark:keytopehalftrue}
\emph{Kohnert's rule extends to key polynomials (with proof). Hence a similar argument (which we omit) 
establishes the ``$\supseteq$'' containment of Conjecture~\ref{conj:keytopeineq}.}\qed
\end{Remark}

Fix a partition $\lambda=(\lambda_1,\lambda_2,\ldots,\lambda_n)$.
Let ${\sf D}_{\lambda}$
be the Young diagram for $\lambda$ (in French notation) placed
flush left in $n\times n$ (hence row $n$ has $\lambda_1$ boxes).

\begin{Proposition}[The Schubitope is a generalized permutahedron]
\label{thm:first}
${\mathcal S}_{{\sf D}_{\lambda}}={\mathcal P}_{\lambda}\subset {\mathbb R}^n$.
\end{Proposition}

\begin{Lemma}
\label{prop:symmetry}
If $w(i) < w(i+1)$, $\mathcal{S}_{{\sf D}_{w}}$ is symmetric about $i$ and $i+1$.
That is, 
\[(\alpha_1,\alpha_2,\ldots,\alpha_i,\alpha_{i+1},\ldots,\alpha_n)\in 
{\mathcal S}_{{\sf D}_w} \iff 
(\alpha_1,\alpha_2,\ldots,\alpha_{i+1},\alpha_{i},\ldots,\alpha_n)\in 
{\mathcal S}_{{\sf D}_w}.\] 
\end{Lemma}
\begin{proof}
Suppose $S \subseteq [n]$ such that $i \in S, i+1 \notin S$.  Let $S'$ be the set formed from $S$ by replacing $i$ with $i+1$.  
Then it suffices to show for any column $c$, 
\[\theta_{{\sf D}_w}^c(S) = \theta_{{\sf D}_w}^c(S').\]

Since $w(i) < w(i+1)$, if $(i,c) \in {\sf D}_w$, then $(i+1,c) \in {\sf D}_w$ as well.  
There are three cases:

\noindent
{\sf Case 1:} ($(i,c), (i+1,c) \in {{\sf D}_w}$): in ${\tt word}_{c,S}({\sf D}_w)$, 
rows $i,i+1$ contributes $\star)$ whereas in ${\tt word}_{c,S'}({\sf D}_w)$ the
contribution is $)\star$.  The 
$)$ does not change whether or not it is paired and thus $\theta_{{\sf D}_w}^c(S) = 
\theta_{{\sf D}_w}^c(S')$.

\noindent
{\sf Case 2:} ($(i,c) \notin {\sf D}_w, (i+1,c) \in {\sf D}_w$): in ${\tt 
word}_{c,S}({\sf D}_w)$, rows $i,i+1$ contributes $()$. In ${\tt word}_{c,S'}({\sf 
D}_w)$, the contribution is only `$\star$'.  Both contribute $1$ to $\theta_{{\sf 
D}_w}^c(S)$ and 
$\theta_{{\sf D}_w}^c(S')$ respectively. Hence $\theta_{{\sf D}_w}^c(S) = 
\theta_{{\sf D}_w}^c(S')$.

\noindent
{\sf Case 3:} ($(i,c), (i+1,c) \notin {\sf D}_w$): 
in both ${\tt word}_{c,S}({\sf D}_w)$ for rows $i,i+1$ and 
${\tt word}_{c,S'}({\sf D}_w)$, rows $i$ and $i+1$ contribute
$($.  The $($ does not change whether or not it is paired and so $\theta_{{\sf 
D}_w}^c(S) 
= \theta_{{\sf D}_w}^c(S')$.
\end{proof}

\noindent
\emph{Proof of Proposition~\ref{thm:first}:}
By Proposition~\ref{prop:generalthing}(I),
\begin{equation}
\label{eqn:mar4xyz}
{\sf Newton}(s_{\lambda}(x_1,\ldots,x_n))={\mathcal 
P}_{\lambda}\subseteq {\mathbb R}^n
\end{equation}

Let $w_{\lambda,n}$ be the Grassmannian permutation associated to $\lambda$. This 
permutation only has descent at position $n$. Then \begin{equation} \label{eqn:mar4www} 
{\mathfrak S}_{w_{\lambda,n}}=s_{\lambda}(x_1,\ldots,x_n). \end{equation} We next show 
that \begin{equation} \label{eqn:mar4hhh} {\mathcal S}_{{\sf D}_{w_{\lambda,n}}}={\sf 
Newton}({\mathfrak S}_{w_{\lambda,n}}). \end{equation} The ``$\supseteq$'' containment 
of (\ref{eqn:mar4hhh}) is given by Proposition~\ref{thm:onecontainment}. 
In the case at hand, this proposition can be deduced from 
A.~Kohnert's work \cite{Kohnert} 
who proved his conjecture for Grassmannian permutations.
Below we will use that in \emph{loc cit.}, A.~Kohnert  proved the Grassmannian case by 
giving a weight-preserving bijection 
$\phi:{\sf SSYT}(\lambda,[n])\to {\tt Koh}(w_{\lambda,n})$, 
where
${\sf SSYT}(\lambda,[n])$ 
is the set of semistandard tableaux of shape
$\lambda$ with fillings using $1,2,\ldots,n$.

We now obtain the other containment of (\ref{eqn:mar4hhh}).
Let 
$(\alpha_1,\alpha_2,\ldots,\alpha_n)\in {\mathcal S}_{{\sf D}_{w_{\lambda,n}}}$. 
In fact, ${\sf D}_{w_{\lambda,n}}$ 
differs from ${\sf D}_{\lambda}$ by removing empty 
columns and left justifying. Hence
it is clear from the definition of 
${\theta}_{{\sf D}_{\lambda}}(S)$ that
\begin{equation}
\label{eqn:mar4theineqabc}
\sum_{i=1}^t \alpha_i \leq \sum_{i=1}^t \lambda_i \text{\ for $t=1,\ldots,n$.}
\end{equation}

Lemma~\ref{prop:symmetry} implies that ${\sf D}_{w_{\lambda,n}}$ has an 
$S_n$-action by permutation of the coordinates. Hence if $\beta=\lambda(\alpha)$ is the
decreasing rearrangement of $\alpha$, then $\beta$ also satisfies 
(\ref{eqn:mar4theineqabc}), where $\beta$ replaces $\alpha$. That is,
$\beta\leq_D\lambda$. 

Therefore by (\ref{eqn:kostkadominance}), $K_{\lambda,\beta}\neq 0$ and there exists a 
semistandard
tableau of shape $\lambda$ and content $\beta$. By symmetry of 
$s_{\lambda}(x_1,\ldots,x_n)$ (and the fact it is the weight-generating series for
${\sf SSYT}(\lambda,[n])$), there is a semistandard
tableau $U$ of shape $\lambda$ and content $\alpha$. 

Now apply Kohnert's bijection $\phi$ to contain $D\in {\tt Koh}(w_{\lambda,n})$ with 
${\tt Kohwt}(D)=\alpha$, as desired. This completes the proof of (\ref{eqn:mar4hhh}).

Since ${\sf D}_w$ and ${\sf D}_{\lambda}$ only differ by a column permutation 
${\mathcal S}_{{\sf D}_{\lambda}}= {\mathcal S}_{{\sf D}_{w_{\lambda,n}}}$.
Now combine this with (\ref{eqn:mar4hhh}), (\ref{eqn:mar4www}) and (\ref{eqn:mar4xyz}).
\qed

The above result can be also deduced by comparing the inequalities of ${\mathcal S}_{{\sf D}_{\lambda}}$ with those for
${\mathcal P}_{\lambda}$. However, the above argument has elements that might apply more generally.

\section*{Acknowledgements}
We thank Alexander Barvinok, Laura Escobar, Sergey Fomin, Allen Knutson, Melinda Lanius, Fu Liu, Mark Shimozono, John Stembridge,
Sue Tolman and Anna
Weigandt for very helpful conversations. 
We thank Bruce Reznick specifically for his example of $f=x_1^2+x_2 x_3+\cdots$ we used in the introduction. AY was supported by an NSF grant.
CM and NT were supported by UIUC Campus Research Board Grants. We made significant use 
of SAGE during our investigations.

\end{document}